\documentclass[12pt, twoside, leqno]{article}

\usepackage{amsmath,amsthm}
\usepackage{amssymb}
\usepackage{enumerate}
\usepackage[toc,page]{appendix}

\usepackage{graphicx}

\newtheorem{theorem}{Theorem}[section]
\newtheorem{cor}[theorem]{Corollary}
\newtheorem{lemma}[theorem]{Lemma}
\newtheorem{prop}[theorem]{Proposition}
\newtheorem{conj}[theorem]{Conjecture}

\newtheorem{mainthm}[theorem]{Main Theorem}

\theoremstyle{definition}

\newtheorem{remark}[theorem]{Remark}

\numberwithin{equation}{section}

\DeclareMathOperator{\res}{Res}
\DeclareMathOperator{\spn}{span}
\newcommand{\R}{\mathbb{R}}
\newcommand{\Z}{\mathbb{Z}}

\newcommand{\C}{\mathbb{C}}
\newcommand{\uple}[1]{\text{\boldmath${#1}$}}

\begin{document}


\baselineskip=17pt


\title{The shifted fourth moment of automorphic $L$-functions of prime power level}

\author{Olga Balkanova\\
Institut de Math\'{e}matiques de Bordeaux\\
Universit\'{e} de Bordeaux\\
351 cours de la Liberation,
33405 Talence cedex,
France\\
E-mail: olga.balkanova@math.u-bordeaux1.fr}

\date{}

\maketitle


\renewcommand{\thefootnote}{}



\renewcommand{\thefootnote}{\arabic{footnote}}
\setcounter{footnote}{0}


\begin{abstract}
We prove the asymptotic formula for the fourth moment of automorphic $L$-functions of level  $p^{\nu}$, where $p$ is a fixed prime number and $\nu \rightarrow \infty$.  This paper is a continuation of work by Rouymi, who computed asymptotics of the first three moments at prime power level, and a generalization of results obtained for prime level by Duke, Friedlander \& Iwaniec  and Kowalski, Michel \& Vanderkam.
\end{abstract}
\section{Introduction}

Let $L(s,f)$ be an automorphic $L$-function associated to  $f \in H_{k}^{*}(q)$, where $H_{k}^{*}(q)$ denote the set of primitive forms of weight $k$ and level $q$.
An important subject in analytic number theory is the behavior of such $L$-functions near the critical line.
Questions of particular interest are subconvexity bounds and proportion of  non-vanishing $L$-values. A possible way to analyze these problems is the method of moments.
This technique proved to be very effective in the recent years; see \cite{Du}, \cite{DFI}, \cite{IwSa},  \cite{ KMV5}, \cite{KMV} for details and examples.

In 1995 Duke \cite{Du} proved the asymptotic formula for the first moment and the upper bound for the second moment when $q$ is prime and $k=2$. Four years later, Akbary \cite{Akbary} generalized this result to the case of prime $q$  and $k>2$.
In 2011 Ichihara \cite{Ichihara} found the asymptotic formula for the first moment when $q$ is a power of prime
and $2\leq k \leq 10,$ $k=14$. In the same year, Rouymi \cite{R} computed asymptotics of the first, second and third moments when
$q$ is a power of prime and $k$ is an arbitrary fixed even integer.

In order to break the convexity barrier, one needs to evaluate the limit moment of order
\begin{equation}
\kappa_0:=\liminf_{q \rightarrow \infty}4\frac{\log{|H_{k}^{*}(q)|}}{\log{q}}=4.
\end{equation}
At the same time, starting from the $\kappa_0$-th moment the main term of asymptotic formula contains nontrivial non-diagonal contribution. See \cite{MR2019017} for details.

The moment of order $\kappa_0=4$ for prime $q$,  $q \rightarrow \infty$ was studied by Duke, Friedlander \& Iwaniec \cite{DFI} and Kowalski, Michel \& Vanderkam \cite{KMV}. The main term of the fourth moment splits into diagonal $M^{D}$, off-diagonal $M^{OD}$ and off-off-diagonal $M^{OOD}$ parts. Therefore, it requires three different stages of analysis.
\begin{theorem}\label{mic}(\cite{KMV}, corollary 1.3)
Let $q$ be a prime number and $k=2$.
For all $\epsilon>0$
\begin{equation}
\sum_{f \in H_{k}^{*}(q)}^{h}L(1/2,f)^4=R(\log{q})+O_{\epsilon}(q^{-1/12+\epsilon}),
\end{equation}
where $R$ is a polynomial of degree $6$ and the leading coefficient is $\frac{1}{60 \pi^2}$.
\end{theorem}

In this paper, the result of theorem \ref{mic} is extended as follows.
\begin{itemize}
\item
We consider the level  of the form $q=p^{\nu}$, where $p$ is a fixed prime number and $\nu \rightarrow \infty$.
\item
We assume that the weight $k\geq 2$ is an arbitrary fixed even integer.
\item
We slightly shift each $L$-function in the product from the critical line $\Re{s}=1/2$
\begin{equation*}
M_4(\uple{t},\uple{r})= \sum_{f  \in H_{k}^{*}(q)}^{h} |L( 1/2+t_1+ir_1, f)|^2 |L(1/2+t_2+ir_2,f)|^2,
\end{equation*}
where $\uple{t}=(t_1,t_2)$, $\uple{r}=(r_1,r_2)$, $t_1,t_2,r_1,r_2 \in  \mathbb{R}$ and $|t_1|, |t_2|<\frac{1}{\log{q}}$.
\end{itemize}

The shifts simplify analysis of the off-off-diagonal term, reveal more clearly a combinatorial structure of mean values and allow us to verify random matrix theory conjectures including all lower order terms by Conrey, Farmer, Keating, Rubinstein and Snaith \cite{Conrey}.

We introduce the following notations
\begin{equation}\hat{q}=\frac{\sqrt{q}}{2\pi},
\end{equation}
\begin{equation}\zeta_q(s)=\zeta(s)(1-p^{-s}).
\end{equation}
\begin{conj}\label{conj:random}(analog of conjectures $4.5.1$ and $4.5.2$ in \cite{Conrey})
Let $q=p^{\nu}$, where $p$ is a fixed prime  and $\nu \geq 3$.
Let $k>0$ be an even integer.
Up to an error term, we have
\begin{multline*}
M_4(\uple{t},\uple{r})=
\frac{\phi(q)}{q}\hat{q}^{-2t_1-2t_2}\sum_{\substack{\epsilon_1,\epsilon_2,\epsilon_3,\epsilon_4=\pm 1\\ \epsilon_1\epsilon_2\epsilon_3\epsilon_4=1}}
\hat{q}^{t_1(\epsilon_1+\epsilon_2)+t_2(\epsilon_3+\epsilon_4)+ir_1(\epsilon_1-\epsilon_2)+ir_2(\epsilon_3-\epsilon_4)}
\\ \times \left( \frac{\Gamma(-t_1-ir_1+k/2)\Gamma(-t_1+ir_1+k/2)\Gamma(-t_2-ir_2+k/2)}{\Gamma(t_1+ir_1+k/2)\Gamma(t_1-ir_1+k/2)\Gamma(t_2+ir_2+k/2)}\right)^{1/2}\\
\times \left(\frac{\Gamma(-t_2+ir_2+k/2)\Gamma(\epsilon_1(t_1+ir_1)+k/2)\Gamma(\epsilon_2(t_1-ir_1)+k/2)}{\Gamma(t_2-ir_2+k/2)\Gamma(-\epsilon_1(t_1+ir_1)+k/2)\Gamma(-\epsilon_2(t_1-ir_1)+k/2)}\right)^{1/2}
\\
\times \left( \frac{
\Gamma(\epsilon_3(t_2+ir_2)+k/2)\Gamma(\epsilon_4(t_2-ir_2)+k/2)}{
\Gamma(-\epsilon_3(t_2+ir_2)+k/2)\Gamma(-\epsilon_4(t_2-ir_2)+k/2)}\right)^{1/2}\\
\times \frac{\zeta_{q}(1+t_1(\epsilon_1+\epsilon_2)+ir_1(\epsilon_1-\epsilon_2))\zeta_{q}(1+t_2(\epsilon_3+\epsilon_4)+ir_2(\epsilon_3-\epsilon_4))}
{\zeta_q(2+t_1(\epsilon_1+\epsilon_2)+t_2(\epsilon_3+\epsilon_4)+ir_1(\epsilon_1-\epsilon_2)+ir_2(\epsilon_3-\epsilon_4))}\\
\times \zeta_q(1+\epsilon_1(t_1+ir_1)+\epsilon_3(t_2+ir_2))\zeta_q(1+\epsilon_1(t_1+ir_1)+\epsilon_4(t_2-ir_2))\\
\times\zeta_q(1+\epsilon_2(t_1-ir_1)+\epsilon_3(t_2+ir_2))\zeta_q(1+\epsilon_2(t_1-ir_1)+\epsilon_4(t_2-ir_2)).
\end{multline*}

\end{conj}

\begin{mainthm}\label{mainthm}
Let $\theta=7/64$, $q=p^{\nu}$, where $p$ is a fixed prime  and $\nu \geq 3$.
Let $k>0$ be an even integer.
For all $\epsilon>0$
the fourth moment can be decomposed as follows
\begin{equation*}
M_4(\uple{t},\uple{r})=M^{D}+M^{OD}+M^{OOD}+O_{\epsilon, p, k}\left(q^{\epsilon}\left(q^{-\frac{k-1-2\theta}{8-8\theta}}+q^{-1/4}\right)\right),
\end{equation*}
where the implied constant depends polynomially on $r_1$, $r_2$.
Furthermore,
\begin{multline}\label{dod}
M^D+M^{OD}=\frac{\phi(q)}{q} \sum_{\epsilon_1, \epsilon_2=\pm 1}\hat{q}^{-2t_1-2t_2+2\epsilon_1t_1+2\epsilon_2t_2}  \frac{\Gamma(\epsilon_1t_1+ir_1+k/2)}{\Gamma(t_1+ir_1+k/2)}\\
\times \frac{\Gamma(\epsilon_1t_1-ir_1+k/2)}{\Gamma(t_1-ir_1+k/2)}   \frac{\Gamma(\epsilon_2t_2+ir_2+k/2)\Gamma(\epsilon_2t_2-ir_2+k/2)}{\Gamma(t_2+ir_2+k/2)\Gamma(t_2-ir_2+k/2)}
\\
\times \zeta_q(1+2\epsilon_1t_1) \zeta_q(1+2\epsilon_2t_2)
\frac{\prod_{\epsilon_3,\epsilon_4=\pm 1} \zeta_q(1+\epsilon_1t_1+\epsilon_2t_2+ i\epsilon_3r_1+ i\epsilon_4 r_2) }{\zeta_q(2+2\epsilon_1t_1+2\epsilon_2t_2)}
\end{multline}
and
\begin{multline}\label{ooood}
M^{OOD}=\frac{\phi(q)}{q}\sum_{\epsilon_1,\epsilon_2=\pm 1}\hat{q}^{-2t_1-2t_2+2i\epsilon_1r_1+2i\epsilon_2r_2} \\
\times
\frac{\Gamma(k/2-t_1+i\epsilon_1r_1)\Gamma(k/2-t_2 +i\epsilon_2r_2)}{\Gamma(k/2+t_1-i\epsilon_1r_1)\Gamma(k/2+t_2-i\epsilon_2r_2)}
\\
\times\zeta_q(1+2i\epsilon_1r_1)\zeta_q(1+2i\epsilon_2r_2) \frac{\prod_{\epsilon_3,\epsilon_4= \pm 1} \zeta_q(1+\epsilon_3 t_1+\epsilon_4 t_2+i\epsilon_1r_1+i\epsilon_2r_2)}{\zeta_q(2+2i\epsilon_1r_1+2i\epsilon_2r_2)}.
\end{multline}
\end{mainthm}

\begin{remark}
The condition $\epsilon_1,\epsilon_2,\epsilon_3,\epsilon_4=\pm 1$, $\epsilon_1\epsilon_2\epsilon_3\epsilon_4=1$ in conjecture \ref{conj:random} implies that there are eight terms in the sum.
The four of them
$$(\epsilon_1,\epsilon_2,\epsilon_3,\epsilon_4)=(1,1,1,1),(1,1,-1,-1),(-1,-1,1,1),(-1,-1,-1,-1)$$ coincide with the summands of \eqref{dod}, and the other four
$$(\epsilon_1,\epsilon_2,\epsilon_3,\epsilon_4)=(-1,1,-1,1),(-1,1,1,-1),(1,-1,-1,1),(1,-1,1,-1)$$
with the summands of \eqref{ooood}.
\end{remark}

By letting the shifts tend to zero in theorem \ref{mainthm}, we obtain the asymptotic formula for the fourth moment at the critical point $s=1/2$.
\begin{cor}\label{limitthm}
Let $\theta=7/64$, $q=p^{\nu}$, where $p$ is a fixed prime  and $\nu \geq 3$.
Let $k>0$ be an even integer.
For all $\epsilon>0$
\begin{multline}\label{maineq}
M_4(\uple{0},\uple{0})=\sum_{f \in H_{k}^{*}(q)}^{h}L(1/2,f)^4\\=R(\log{q})+O_{\epsilon,k,p}\left(q^{\epsilon}\left(q^{-\frac{k-1-2\theta}{8-8\theta}}+q^{-1/4}\right)\right),
\end{multline}
where $R$ is a polynomial of degree $6$ and the leading coefficient is
\begin{equation}
\left( \frac{\phi(q)}{q}\right)
^7\frac{p^2}{p^2-1}\frac{1}{60 \pi^2}.
\end{equation}
\end{cor}
The paper is organized as follows. In section \ref{background} we remind the reader of some definitions and fundamental results. Section \ref{preliminary} provides the explicit formula for the diagonal, off-diagonal and off-off-diagonal  main terms. Asymptotics of the diagonal and off-diagonal terms is derived in section \ref{asympD}.
Sections \ref{explicitformula} and \ref{asympOOD} are devoted to proving the asymptotic formula for the off-off-diagonal term.
Corollary \ref{limitthm} is proved as a limit case at the end of sections \ref{asympD} and \ref{asympOOD}.

\section{Background information}\label{background}
The purpose of this section is to recall some results on automorphic forms and related subjects.
\subsection{Automorphic $L$-functions}
 A holomorphic function $f$ on the Poincar\'{e} upper-half plane $\mathbb{H}=\{z \in \mathbb{C}, \Im{z}>0\}$ is called a
 \emph{cusp form} of weight $k$ and level $q$ if it satisfies the following conditions:
\begin{equation}f(\gamma z)=(cz+d)^kf(z)
\end{equation} for all
$\gamma \in \Gamma_0(q)=\left \{
\left(\begin{array}{lcr}
a& b\\
c &d
\end{array}\right)\in SL(2,\mathbb{Z} )\text{ such that } c \equiv 0\pmod{q}\right \},$

\begin{equation}
(\Im{z})^{k/2}|f(z)|\text{ is bounded on }\mathbb{H}.
\end{equation}
Let $S_k(q)$ be the space of cusp forms of weight $k \geq 2$ and level $q$. It is equipped with the Petersson inner product
\begin{equation} \label{Pet}
\langle f,g\rangle_q:=\int_{F_0(q)}f(z)\overline{g(z)}y^{k}\frac{dxdy}{y^2},
\end{equation}
where $F_0(q)$ is a fundamental domain of the action of $\Gamma_0(q)$ on $\mathbb{H}$.
Any $f \in S_{k}(q)$ has a Fourier expansion at infinity
\begin{equation}f(z)=\sum_{n\geq 1}a_f(n)e(nz).
\end{equation}
According to the Atkin-Lehner theory \cite{atkin-lehner}, the space $S_k(q)$ can be decomposed into two subspaces
\begin{equation}
S_k(q)=S_{k}^{new}(q)\oplus S_{k}^{old}(q).
\end{equation}
\emph{The space of old forms} contains cusp forms of level $q$ coming from lower levels
\begin{equation}
S_{k}^{old}(q)=\spn\left\{f(lz): lq'| q, q'<q,f(z)\in S_k(q')\right\},
\end{equation}
and \emph{the space of new forms} is the orthogonal compliment to $S_{k}^{old}(q)$.

We let $H_{k} (q)$ denote an orthogonal basis of the space of cusp forms $S_k(q)$ and $H_{k}^{*}(q)$- an orthogonal basis of $S_{k}^{new}(q)$. Elements of $H_{k}^{*}(q)$ with normalized Fourier coefficients
\begin{equation}
\lambda_f(n):=a_f(n)n^{-(k-1)/2},
\end{equation}
\begin{equation}
\lambda_f(1)=1
\end{equation}
are called \emph{primitive forms}. Accordingly,
\begin{equation}
\lambda_f(n)\in \R,
\end{equation}
\begin{equation}\label{multPR}
\lambda_f(n_1)\lambda_f(n_2)=\sum_{\substack{d|(n_1,n_2)\\ (d,q)=1} } \lambda_f\left( \frac{n_1n_2}{d^2} \right).
\end{equation}
Let $Re (s)>1$. Then for $f \in
H_{k}^{*}(q)$ we define the \emph{automorphic $L$-function} as
\begin{equation}
L(s,f)=\sum_{n \geq 1}\lambda_f(n)n^{-s}.
\end{equation}
The \emph{completed $L$-function}
\begin{equation}
\Lambda(s,f)=\left(\frac{\sqrt{q}}{2 \pi}\right)^s\Gamma\left(s+\frac{k-1}{2}\right)L(s,f)
\end{equation}
can be analytically continued on the whole complex plane and satisfies the functional equation
\begin{equation} \label{eq: functionalE}
\Lambda(s,f)=\epsilon_f\Lambda(1-s,f),
\end{equation}
where $s\in \C$ and $\epsilon_f=\pm1$.
We define the \emph{harmonic average} over the set of primitive newforms by
\begin{equation} \label{harmonicAv}
\sum_{f \in H^{*}_{k}(q)}^{h}A(f):=\sum_{f \in
H^{*}_{k}(q)}\frac{\Gamma(k-1)}{(4\pi)^{k-1} \langle f,f\rangle_q}A(f).
\end{equation}

\subsection{Kloosterman sums}
Consider the sum
\begin{equation}\label{klsum}
S(m,n,c)=
 \sum_{\substack{
   d\text{(mod c)} \\
   (c,d)=1
  }}
 e\left(\frac{m\overline{d}+nd}{c}\right),
\end{equation} where  $d\overline{d}\equiv 1\text{(mod c)}$, $e(z)=exp(2\pi iz).$

The value of $S(m,n,c)$ is always a real number because
\begin{equation}\overline{S(m,n,c)}=S(m,n,c).
\end{equation}
Further,
 \begin{equation}
 S(m,n,c)=S(n,m,c),
 \end{equation}
 \begin{equation}
 S(ma,n,c)=S(m,na,c)\text{ if }(a,c)=1.
 \end{equation}
Another important property is the twisted multiplicity (\cite{Iw1} formula (4.12)).
Suppose $(c_1,c_2)=1$,
$c_2 \overline{c_2} \equiv 1\text{(mod }c_1),$
$c_1\overline{c_1}\equiv 1\text{(mod }c_2).$
Then
 \begin{equation}
 S(m,n,c_1c_2)=S(m\overline{c_2},n\overline{c_2},c_1)S(m\overline{c_1},n\overline{c_1},c_2).
 \end{equation}
\begin{lemma}(Weil's bound, \cite{weil})
One has
\begin{equation}\label{eq:Weil}
|S(m,n,c)|\leq (m,n,c)^{1/2}c^{1/2}\tau(c).
\end{equation}
\end{lemma}
\begin{lemma}(Royer, \cite{royer}, lemma A.12)\label{Royer}
Let $m,n,c$ be three strictly positive integers and $p$ be a prime
number. Suppose $p^2|c$, $p|m$ and $p \nmid n$. Then $S(m,n,c)=0$.
\end{lemma}

\subsection{Large sieve inequality}
\begin{theorem}\label{thm:DesIw}(Deshouillers, Iwaniec, theorem $9$ of \cite{DesIw})
Let $r$ and $s$ be positive coprime integers, $C$, $M$, $N$ be positive real numbers and $g$ be real-valued function of $\mathbf{C^6}$ class (first and second derivatives are continuous for each of variables) with support in $[M,2M]\times [N,2N] \times [C,2C]$ such that
\begin{equation}\label{LSEcond}
\left| \frac{\partial^{(j+k+l)}}{\partial m^{(j)} \partial n^{(k)} \partial c^{(l)}}g(m,n,c)\right| \leq M^{-j}N^{-k}C^{-l} \text { for }0 \leq j,k,l \leq 2.
\end{equation}
Then for any $\epsilon>0$ and complex sequences $\uple{a}=\{a_m\}$, $\uple{b}=\{b_n\}$ one has
\begin{multline}
\sum_{(c,r)=1} \sum_{m}a_m\sum_{n}b_n g(m,n,c)S(m\bar{r},\pm n,sc)  \ll_{\epsilon}\\
 \left(\sum_{M< m\leq 2M }|a_m|^2\right)^{1/2} \left(\sum_{N< n\leq 2N}|b_n|^2\right)^{1/2}C^{\epsilon} \left( 1+\frac{s\sqrt{r}C}{\sqrt{MN}}\right)^{2 \theta}
\\ \times \frac{(s\sqrt{r}C+\sqrt{MN}+\sqrt{sM}C)(s\sqrt{r}C+\sqrt{MN}+\sqrt{sN}C)}{s\sqrt{r}C+\sqrt{MN}}.
\end{multline}
\end{theorem}
Here  $$\theta=\theta_{rs}:=\sqrt{\max{(0,1/4-\lambda_1)}}$$ and $\lambda_1=\lambda_1(rs)$ is the smallest
positive eigenvalue for the Hecke congruence subgroup $\Gamma_0(rs)$.
Currently the best known bound on $\lambda_1$ is  due to Kim and Sarnak \cite{KS}. Accordingly, we can take
$\theta=7/64.$

\subsection{Petersson's trace formula}\label{Section:Pet}
The key ingredient of our proof is the Petersson trace formula. It allows expressing average of Fourier coefficients of cusp forms in terms of Kloosterman sums weighted by $J$-Bessel functions.
\begin{theorem}\label{Petersson}(proposition 14.5, \cite{IwKow})
For $m,n \geq 1$ we have
\begin{multline}
\Delta_{q}(m,n):=\sum_{f\in H_{k} (q)}^{h}
\lambda_f(m)\lambda_f(n)\\=\delta_{m,n}+ 2\pi i^{-k}\sum_{q | c}
\frac{S(m,n,c)}{c} J_{k-1} \left( \frac{4\pi\sqrt{mn}}{c} \right).
\end{multline}
\end{theorem}

If $q$ is a prime number and $k<12$, the Petersson trace formula also works for moments of $L$-functions associated to primitive forms because the space of old forms is empty.

When $q$ is composite, one needs to exclude the contribution of old forms.
Iwaniec, Luo and Sarnak constructed a special basis in order to find an analog of Petersson's trace formula for primitive forms of square-free level.
\begin{theorem}(proposition 2.8 of \cite{ILS})
Let $q$ be square-free, $(m,q)=1$ and $(n,q^2)|q$. Then
\begin{multline}
\Delta_{q}^{*}(m,n):=\sum_{f\in H_{k}^{*}(q)}^{h}
\lambda_f(m)\lambda_f(n)\\ =\frac{k-1}{12}\sum_{LM=q}\frac{\mu(L)M}{(n,L)\prod_{p|(n,L)}(1+p^{-1})}\sum_{l|L^{\infty}}l^{-1}\Delta_M(ml^2,n).
\end{multline}
\end{theorem}
This result was extended to the case of prime power level by Rouymi.
\begin{theorem}\label{PetRou}(remark $4$ of \cite{R})
Let  $q=p^{\nu}$, $\nu \geq 3$. Then
\begin{multline}\label{eq:rtrace}
\Delta_{q}^{*}(m,n):=\sum_{f\in H_{k}^{*}(q)}^{h}
\lambda_f(m)\lambda_f(n)\\=\begin{cases}
\Delta_{q}(m,n)-\frac{\Delta_{q/p}(m,n)}{p}& \text{ if } (q,mn)=1,\\
0& \text{otherwise}.
\end{cases}
\end{multline}
 \end{theorem}
\subsection{Poisson type summation formula connected with the Eisenstein-Maass series}

Let
\begin{equation}\tau_v(n)=|n|^{v-1/2}\sigma_{1-2v}(n)=|n|^{v-1/2}\sum_{d|n, d>0}d^{1-2v}.
\end{equation}

If $v=1/2$, then $\tau_v(n)$ reduces to the divisor function $\tau(n
)$.
Furthermore, $\tau_v(n)$ satisfies the property of multiplicity (see \cite{Kuz}, page 74)
\begin{equation}\label{multtau}
\tau_{v}(n)\tau_{v}(m)=\sum_{d|(n,m)}\tau_{v}\left(\frac{nm}{d^2}\right).
\end{equation}
\begin{lemma}(Ramanujan's identity, \cite{titc}, page 8)

Let $\Re{s}>1+|\Re{v}-1/2|+|\Re{\mu}-1/2|$. Then
\begin{multline}\label{ramshift}
\zeta(2s)\sum_{n \geq 1}\frac{\tau_{v}(n)\tau_{\mu}(n)}{n^s}=\\
\zeta(s+v-\mu)\zeta(s-v+\mu)\zeta(s+v+\mu-1)\zeta(s-v-\mu+1).
\end{multline}
If $v=\mu=1/2$, this reduces to
 \begin{equation}\label{Ramid}
 \sum_{n\geq 1}\frac{\tau(n)^2}{n^{s}}=\frac{\zeta(s)^4}{\zeta(2s)}.
 \end{equation}
\end{lemma}
Consider the Bessel kernels expressed in terms of $J$ and $K$-Bessel functions
\begin{equation}
k_0(x,v):=\frac{1}{2\cos{\pi v}}(J_{2v-1}(x)-J_{1-2v}(x)),
\end{equation}

\begin{equation}
k_1(x,v):=\frac{2}{\pi}\sin{\pi v}K_{2v-1}(x).
\end{equation}

\begin{theorem}\label{kuzth}(theorem 5.2 of \cite{Kuz}, page 89)
Let $\phi$ be a smooth, compactly supported function on $\mathbb{R}^+$.
Then for every $v$ with $\Re{v}=1/2$, $(c,d)=1$, $c \geq 1$ one has
\begin{multline}\label{kuzth1}
\frac{4\pi}{c}\sum_{m \geq 1}e\left(\frac{md}{c}\right)\tau_v(m)\phi \left(\frac{4\pi\sqrt{m}}{c}\right)=\\ 2\frac{\zeta(2v)}{(4\pi)^{2v}}\hat{\phi}(2v+1)+2\frac{\zeta(2-2v)}{(4\pi)^{2-2v}}\hat{\phi}(3-2v)+\\ \sum_{m \geq 1}\tau_v(m)\int_{0}^{\infty}\left[e\left(-\frac{ma}{c} \right)k_0(x \sqrt{m},v) +e\left(\frac{ma}{c} \right)k_1(x \sqrt{m},v)\right]\phi(x)xdx,
\end{multline}

where $ad \equiv 1 \pmod{c}$ and $\hat{\phi}$ is the Mellin transform of $\phi$.

\end{theorem}
\subsection{Quadratic divisor problem}
Applying formula \eqref{kuzth1}, we generalize theorem $1$ of \cite{DFIQ} as follows.
\begin{theorem}\label{thm:DFI1}

Let $ a,b \geq 1$, $(a,b)=1$, $h \neq 0$, $r_1,r_2 \in \R$.
Let
$$D_f(a,b;h)=\sum_{am \mp bn=h}\tau_{1/2+ir_1}(m)\tau_{1/2+ir_2}(n)f(am,bn)$$ with
\begin{equation}\label{eq:condition1}
x^iy^jf^{(ij)}(x,y)\ll (1+\frac{x}{X})^{-1}(1+\frac{y}{Y})^{-1}Q^{i+j}.
\end{equation}
Assume that
\begin{equation}
ab<Q^{-5/4}(X+Y)^{-5/4}(XY)^{1/4+\epsilon}.
\end{equation}
Then $$D_f(a,b;h)=\int_{0}^{\infty}g(x, \pm x \mp
h)dx+O(Q^{5/4}(X+Y)^{1/4}(XY)^{1/4+\epsilon}),$$
where the implied constant depends polynomially on $r_1$, $r_2$.
Here $g(x,y)=f(x,y)\Lambda_{a,b,h}(x,y)$ with
\begin{multline}		
\Lambda_{a,b,h}(x,y):=\sum_{w=1}^{\infty}S(0,h,w)\sum_{\epsilon_1,\epsilon_2= \pm 1}\frac{(ab,w)(a,w)^{2i\epsilon_1r_1}(b,w)^{2i\epsilon_2r_2}}{a^{1+i\epsilon_1r_1}b^{1+i\epsilon_2r_2}w^{2+2i\epsilon_1r_1+2i\epsilon_2r_2}}\\ \times \zeta(1+2i\epsilon_1r_1)
\zeta(1+2i\epsilon_2r_2)
x^{i\epsilon_1r_1}y^{i\epsilon_2r_2}.
\end{multline}
\end{theorem}

\section{The fourth moment: preliminary steps}\label{preliminary}

\subsection{Approximate functional equation}

Let $P_r(s)$ be an even polynomial vanishing at all poles of $\Gamma(s+ir+k/2)\Gamma(s-ir+k/2)$ in the range $\Re{s}\geq -L$ for some large constant $L>0$.
For $t, r \in \R$ we define
\begin{multline}\label{eq:W}
W_{t,r}(y):= \frac{1}{2\pi i}\int_{(3)}\frac{P_r(s)}{P_r(t)}\zeta_q(1+2s)\\
\times \frac{\Gamma(s+ir+k/2)\Gamma(s-ir+k/2)}{\Gamma(t+ir+k/2)\Gamma(t-ir+k/2)}y^{-s}\frac{2sds}{s^2-t^2}.
\end{multline}
\begin{lemma}
Suppose $y>0$, $|t|<1/2$. For any $C>|t|$
\begin{equation}\label{eq:Wbound}
W_{t,r}(y) =O_{C,t,r}(y^{-C}) \text{ as }y\rightarrow{\infty},
\end{equation}
\begin{multline}\label{eq:W0bound}
W_{t,r}(y) =\zeta_{q}(1-2t)y^{t}\frac{\Gamma(-t+ir+k/2)\Gamma(-t-ir+k/2)}{\Gamma(t+ir+k/2)\Gamma(t-ir+k/2)}\\+\zeta_{q}(1+2t)y^{-t}
+O_{C,t,r}(y^{C}) \text{ as }y \rightarrow 0.
\end{multline}
The implied constants depend polynomially on $r$.
\end{lemma}
\begin{proof}
Asymptotic expansion for the ratio of gamma functions gives
$$\frac{\Gamma(C+ir+k/2)\Gamma(C-ir+k/2)}{\Gamma(t+ir+k/2)\Gamma(t-ir+k/2)}=(|r|)^{2(C-t)}\left( 1+O(1/|r|)\right).$$
First, without crossing any poles, we can shift the contour of integration to $\Re{s}=C$ with $C>|t|$. This implies \eqref{eq:Wbound}.
Second, we move the contour of integration to $\Re{s}=-C$, meeting two simple poles at $s=\pm t$. Therefore, as $y\rightarrow 0$, we have \eqref{eq:W0bound}.

\end{proof}

\begin{lemma}\label{prop:funce1}
For $t,r \in \R$, $|t|<1/2$ we have
\begin{equation}
|L(1/2+t+ir,f)|^2=
\left(\hat{q}\right)^{-2t}\sum_{n \geq 1}\tau_{1/2+ir}(n)\frac{\lambda_f(n)}{\sqrt{n}}W_{t,r}\left(\frac{n}{\hat{q}^2}\right).
\end{equation}
\end{lemma}
\begin{proof}
Consider
$$I_t:=\frac{1}{2\pi i}\int_{(3)}\Lambda(1/2+s+ir,f)\Lambda(1/2+s-ir,f)\frac{P_r(s)}{s-t}ds.$$
Moving the contour of integration to $\Re{s}=-3$, we pick up a simple pole at $s=t$. The functional equation \eqref{eq: functionalE} implies that
\begin{multline*}I_t+\epsilon_{f}^{2}I_{-t}=Res_{s=t}\left( \Lambda(1/2+s+ir,f) \Lambda(1/2+s-ir,f)\frac{P_r(s)}{s-t}\right)\\=P_r(t)\Lambda(1/2+t+ir,f)\Lambda(1/2+t-ir,f).
\end{multline*}
Observe that for $s>1/2$, the property \ref{multPR} yields
\begin{equation*}
|L(1/2+s+ir,f)|^2=
\zeta_q(1+2s)\sum_{n \geq 1}\frac{\lambda_f(n)}{n^{1/2+s
}}\tau_{1/2+ir}(n).
\end{equation*}
Finally,
\begin{multline*}
|L(1/2+t+ir,f)|^2=
\left(\hat{q}\right)^{-2t}\sum_{n \geq 1}\tau_{1/2+ir}(n)\frac{\lambda_f(n)}{\sqrt{n}}\\ \times
\frac{1}{2\pi i}\int_{(3)}\frac{P_r(s)}{P_r(t)}\zeta_q(1+2s)\frac{\Gamma(s+ir+k/2)\Gamma(s-ir+k/2)}{\Gamma(t+ir+k/2)\Gamma(t-ir+k/2)}\left(\frac{n}{\hat{q}^2}\right)^{-s}\frac{2sds}{s^2-t^2}.
\end{multline*}

\end{proof}

\begin{cor}
The fourth moment can be written as follows
\begin{multline} \label{fourthm}
M_4(\uple{t},\uple{r})=\left(\hat{q}\right)^{-2t_1-2t_2}\sum_{m,n \geq 1}\tau_{1/2+ir_1}(m)\tau_{1/2+ir_2}(n)\\
\times \frac{1}{\sqrt{mn}}W_{t_1,r_1}\left(\frac{m}{\hat{q}^2}\right)W_{t_2,r_2}\left(\frac{n}{\hat{q}^2}\right)\Delta_{q}^{*}(m,n).
\end{multline}
\end{cor}

\subsection{Applying the Petersson trace formula}
Here we apply theorem \ref{PetRou} for $\nu \geq 3$. The case $\nu=2$ can be treated similarly, but doesn't seem to be of particular interest since the final goal is $\nu=\infty$.
 Let \begin{multline} \label{eq:TC}
 T(c):=c\sum_{\substack{m,n \geq 1\\(q,mn)=1}
}\frac{\tau_{1/2+ir_1}(m)\tau_{1/2+ir_2}(n)}{\sqrt{nm}}\\
\times W_{t_1,r_1}\left(\frac{m}{
\hat{q}^2}\right) W_{t_2,r_2}\left( \frac{n}{ \hat{q}^2}\right) S(m,n,c)J_{k-1} \left( \frac{4\pi\sqrt{mn}}{c} \right).
 \end{multline}

Using the trace formula \eqref{eq:rtrace}, the fourth moment \eqref{fourthm} can be written as a sum of diagonal and non-diagonal parts.
\begin{prop}
The following decomposition takes place
\begin{equation}\label{eq:M}
M_4(\uple{t},\uple{r})=M^D+M_{1}^{ND}+M_{2}^{ND},
\end{equation}
 where
\begin{multline} \label{eq:diag}
M^D= \frac{\phi(q)}{q} \hat{q}^{-2t_1-2t_2}\\
\times \sum_{\substack{n \geq 1\\ (q,n)=1}}
\frac{\tau_{1/2+ir_1}(n)\tau_{1/2+ir_2}(n)}{n}W_{t_1,r_1}\left(\frac{n}{\hat{q}^2}\right)W_{t_2,r_2}\left(\frac{n}{\hat{q}^2}\right),
\end{multline}
\begin{equation}\label{eq:ND1}
M_{1}^{ND}=2\pi i^{-k} \hat{q}^{-2t_1-2t_2}\sum_{q|c}
\frac{1}{c^2}T(c),
\end{equation}
and
\begin{equation}\label{eq:ND2}
M_{2}^{ND}=-\frac{2\pi i^{-k}}{ p}\hat{q}^{-2t_1-2t_2}\sum_{\frac{q}{p}|c} \frac{1}{c^2}T(c).
\end{equation}
\end{prop}
\begin{remark}
 For any $\epsilon>0$ we have
$M^D \ll_{\epsilon, \uple{r}} \frac{\phi(q)}{q} q^{\epsilon}.$ The asymptotics of this term will be evaluated in section \ref{off-diagonal}.
\end{remark}

\subsection{Smooth partition of unity and restriction of summations}

Assume that $F_X(x)$ is a compactly supported function in $[X/2, 3X]$ such that for any integral $j \geq 0$
\begin{equation}\label{eq:funcF}
x^jF_{X}^{(j)}(x)\ll _j1.
\end{equation}
We make a smooth dyadic partition of unity (see Appendix A of \cite{RiRo} for details). Accordingly,
 $$\frac{1}{\sqrt{mn}}W_{t_1,r_1}\left(\frac{m}{\hat{q}^2}\right)W_{t_2,r_2}\left(\frac{n}{\hat{q}^2}\right)=\sum_{M,N\geq
1}F_{M,N}(m,n),$$
where the sums on $M$, $N$ are over powers of $2$ and
\begin{equation}
F_{M,N}(m,n):=f_{M,t_1,r_1}(m)f_{N,t_2,r_2}(n),
\end{equation}
\begin{equation}
f_{X,t,r}(x):=\frac{1}{\sqrt{x}}W_{t,r}\left(\frac{x}{\hat{q}^2}\right)F_X(x).
\end{equation}
The term \eqref{eq:TC} can be written as
 \begin{equation} \label{eq:summref}
T(c)=\sum_{M,N \geq 1}T_{M,N}(c),
\end{equation}
 \begin{multline} \label{eq:TMNC}
T_{M,N}(c)=c\sum_{\substack{m,n \\(q,mn)=1} }
\tau_{1/2+ir_1}(m)\tau_{1/2+ir_2}(n)\\
\times S(m,n,c)F_{M,N}(m,n)J_{k-1} \left( \frac{4\pi\sqrt{mn}}{c}\right).
\end{multline}

\begin{lemma}\label{boundfmn}
For any $\alpha \geq |t|$,
\begin{equation}x^{i}\frac{\partial^i}{\partial^i x}f_{X,t,r}(x) \ll_{\alpha,t,r}\frac{1}{\sqrt{X}}\left(\frac{ x}{\hat{q}^2}\right)^{-\alpha} \text{ if } X \gg q^{1+\epsilon}
\end{equation}
and
\begin{equation}x^{i}\frac{\partial^i}{\partial^i x}f_{X,t,r}(x) \ll_{t,r}\frac{1}{\sqrt{X}}\left(\frac{ x}{\hat{q}^2}\right)^{-|t|} \text{ if } X \ll q^{1+\epsilon}.
\end{equation}
\end{lemma}
\begin{proof}
If $X \gg q^{1+\epsilon}$ we use \eqref{eq:Wbound} to get
$$x^{i}\frac{\partial^{i}}{\partial^{i} x}W_{t,r}\left(\frac{ x}{\hat{q}^2}\right)\ll_{\alpha,t,r} \left(\frac{ x}{\hat{q}^2}\right)^{-\alpha}.$$
If $X \ll q^{1+\epsilon}$ we use \eqref{eq:W0bound} to get
$$x^{i}\frac{\partial^{i}}{\partial^{i} x}W_{t,r}\left(\frac{ x}{\hat{q}^2}\right)\ll_{t,r} \left(\frac{ x}{\hat{q}^2}\right)^{-|t|}.$$
Finally, estimate \eqref{eq:funcF} and Leibniz's rule yield the result.
\end{proof}

\begin{prop}\label{prop:A}
For any $\epsilon >0$, any $A>0$ and $l=0,1$
 \begin{equation}
\sum_{\max{(M,N)} \gg q^{1+\epsilon}}\sum_{\frac{q}{p^l}|c}\frac{1}{c^2}T_{M,N}(c)\ll_{\epsilon, A,  \uple{r}} q^{-A}.
\end{equation}

\end{prop}
\begin{proof}
Since $\max{(M,N)} \gg q^{1+\epsilon}$, there are three cases to consider:
\begin{itemize}
\item
$M \gg q^{1+\epsilon}$, $N \ll q^{1+\epsilon}$;
\item
$M \ll q^{1+\epsilon}$, $N \gg q^{1+\epsilon}$;
\item
$M \gg q^{1+\epsilon}$, $N \gg q^{1+\epsilon}$.
\end{itemize}
We prove only the first case:
\begin{multline*}
\sum_{\substack{M \gg q^{1+\epsilon}\\ N \ll q^{1+\epsilon}}}\sum_{\frac{q}{p^l}|c}\frac{1}{c^2}T_{M,N}(c)=
\sum_{\substack{M \gg q^{1+\epsilon}\\ N \ll q^{1+\epsilon}}}\sum_{\substack{c,m,n\\ (q,mn)=1\\ \frac{q}{p^l}|c}}\tau_{1/2+ir_1}(m)\tau_{1/2+ir_2}(n)\\
\times \frac{S(m,n,c)}{c}F_{M,N}(m,n)J_{k-1} \left( \frac{4\pi\sqrt{mn}}{c}\right).
\end{multline*}
The sum over $c$ can be decomposed into two cases
\begin{multline*}
\sum_{q/p^{l}|c}\frac{S(m,n,c)}{c}J_{k-1} \left( \frac{4\pi\sqrt{mn}}{c}\right)= \sum_{\substack{c<\sqrt{mn}\\q/p^{l}|c}}\frac{S(m,n,c)}{c}J_{k-1} \left( \frac{4\pi\sqrt{mn}}{c}\right)\\+\sum_{\substack{c \geq \sqrt{mn}\\q/p^{l}|c}}\frac{S(m,n,c)}{c}J_{k-1} \left( \frac{4\pi\sqrt{mn}}{c}\right).
\end{multline*}
By \eqref{eq:Jbes} and \eqref{eq:Weil} for any $\delta>0$ we have
\begin{equation*}
\sum_{q/p^{l}|c}\frac{S(m,n,c)}{c}J_{k-1} \left( \frac{4\pi\sqrt{mn}}{c}\right)\ll (mn)^{3/4+\delta}.
\end{equation*}
We apply lemma \ref{boundfmn} with $i=j=0$:
$$\sum_{\frac{q}{p^l}|c}\frac{1}{c^2}T_{M,N}(c)\ll_{\alpha_1,\uple{r}} (MN)^{1/4+\delta} \left(\frac{q}{M}\right)^{\alpha_1}\left(\frac{q}{N}\right)^{|t_2|}.$$
Taking $\alpha_1$ sufficiently large,  we obtain that for any $\epsilon >0$, any $A>0$ and $l=0,1$
 \begin{equation*}
\sum_{\substack{M \gg q^{1+\epsilon}\\ N \ll q^{1+\epsilon}} }\sum_{\frac{q}{p^l}|c}\frac{1}{c^2}T_{M,N}(c)\ll_{\epsilon, A,\uple{r}} q^{-A}.
\end{equation*}

 \end{proof}
\begin{cor}
The range of summation in \eqref{eq:summref} can be restricted to $M,N \ll q^{1+\epsilon}$.
\end{cor}

Finally, we restrict the range of summation on $c$ via large sieve inequality.
\begin{lemma}\label{lemma:LSI}
Let $l=0,1$. Assume that $M,N \ll q^{1+\epsilon}$. For any $C>\sqrt{MN}$ we have
 \begin{equation}
\sum_{\substack{c \geq C\\ \frac{q}{p^l}|c}}\frac{1}{c^2}T_{M,N}(c) \ll_{\epsilon,\uple{r}} \left( \frac{\hat{q}^2}{M}\right)^{|t_1|}\left( \frac{\hat{q}^2}{N}\right)^{|t_2|}q^{\epsilon} \left(\frac{\sqrt{MN}}{C}\right)^{k-1-2\theta}.
\end{equation}
\end{lemma}
\begin{remark}
Taking $C=\min\left(q^{\frac{1}{2-2\theta}}M^{1/2}N^{\frac{1-4\theta}{8-8\theta}}, q^{\frac{9-8\theta}{8-8\theta}}\right)$, the error term is  bounded by $q^{\epsilon-\frac{k-1-2\theta}{8-8\theta}}.$
See the proof of lemma  \ref{choice} for the explanation of this choice.
\end{remark}
\begin{proof}
We are going to apply theorem \ref{thm:DesIw}. In order to do so, we make a dyadic partition of the interval $[C,\infty)$ and assume that $c \in [C,2C]$.
By definition
\begin{multline*}
\sum_{\frac{q}{p^l}|c}\frac{1}{c^2}T_{M,N}(c)=
\sum_{\substack{n,m\\ (q,nm)=1}}\sum_{\frac{q}{p^l}|c}\tau_{1/2+ir_1}(m)\tau_{1/2+ir_2}(n)\frac{1}{c}S(m,n,c)\\ \times J_{k-1}\left(\frac{4\pi\sqrt{mn}}{c}\right )F_{M,N}(m,n)
= \frac{p^l}{q} \sum_{\substack{n,m\\ (q,nm)=1}}\tau_{1/2+ir_1}(m)\tau_{1/2+ir_2}(n)\\ \times\sum_{c_1}\frac{1}{c_1}S(m,n,c_1q/p^l)J_{k-1}\left(\frac{4\pi\sqrt{mn}p^l}{c_1q}\right )F_{M,N}(m,n).
\end{multline*}
Here $m \in [M/2,3M]$, $n \in [N/2,3N]$ and $c_1 \in [C_1,2C_1]$ with $C_1:=Cp^l/q$.
Let $$X:=\left(\frac{\hat{q}^2}{M}\right)^{-|t_1|}\left(\frac{\hat{q}^2}{N}\right)^{-|t_2|}\sqrt{MN}C_1\left(\frac{\sqrt{MN}}{C}\right)^{-k+1}.$$
As a test function we choose
$$g(m,n,c_1):=\frac{X}{c_1}F_{M,N}(m,n)J_{k-1}\left(\frac{4\pi\sqrt{mn}p^l}{c_1q}\right).$$
It satisfies condition \eqref{LSEcond},
 and theorem  \ref{thm:DesIw} can be applied with $r=1$ and $s=q/p^l$.
Hence
\begin{equation*}
\sum_{\substack{\frac{q}{p^l}|c\\ c \geq C}}\frac{1}{c^2}T_{M,N}(c) \ll_{\epsilon, \uple{r}}  \left( \frac{\hat{q}^2}{M}\right)^{|t_1|}\left( \frac{\hat{q}^2}{N}\right)^{|t_2|}q^{\epsilon} \left(\frac{\sqrt{MN}}{C}\right)^{k-1-2\theta}.
\end{equation*}
\end{proof}

\subsection{Removing the coprimality condition}
In order to apply theorem \ref{kuzth}, we have to exclude the coprimality condition in $T_{M,N}(c)$.
This can be done using the criterion of vanishing of classical Kloosterman sum given by lemma \ref{Royer}.
Let \begin{equation}\label{fmnc}
f(m,n,c):=F_{M,N}(m,n)J_{k-1}\left(\frac{4\pi\sqrt{mn}}{c}\right).
\end{equation}
\begin{prop}\label{mainprop}
Let $m,n,c$ be three strictly positive integers and $p$ be a prime
number. Suppose $p^2|c$.
Then
\begin{multline*}\sum_{(q,mn)=1}\tau_{1/2+ir_1}(m)\tau_{1/2+ir_2}(n)S(m,n,c)f(m,n,c)=\\
\sum_{m,n \geq 1}\tau_{1/2+ir_1}(m)\tau_{1/2+ir_2}(n)S(m,n,c)f(m,n,c)\\- \tau_{1/2+ir_2}(p)\sum_{m,n \geq 1}\tau_{1/2+ir_1}(m)\tau_{1/2+ir_2}(n)S(m,np,c)f(m,np,c)\\+\sum_{m,n \geq 1}\tau_{1/2+ir_1}(m)\tau_{1/2+ir_2}(n)S(m,np^2,c)f(m,np^2,c).
\end{multline*}
\end{prop}
\begin{proof}
Recall that $q=p^\nu$. Therefore,
$$\sum_{(q,mn)=1}= \sum_{p \nmid mn}=\sum_{m,n}- \sum_{p|mn}=
\sum_{m,n}-\sum_{p|n}-\sum_{p|m,p\nmid n}.$$
The sum $$\sum_{p|m,p\nmid n}\tau_{1/2+ir_1}(m)\tau_{1/2+ir_2}(n)S(m,n,c)f(m,n,c)=0$$
since the Kloosterman sum vanishes by lemma \ref{Royer}. Further,
\begin{multline*}\sum_{p|n}\tau_{1/2+ir_1}(m)\tau_{1/2+ir_2}(n)S(m,n,c)f(m,n,c)\\= \sum_{n}
\tau_{1/2+ir_1}(m)\tau_{1/2+ir_2}(np)S(m,np,c)f(m,np,c).
\end{multline*}
The identity \eqref{multtau}
implies that
$$\tau_{1/2+ir_2}(np)= \tau_{1/2+ir_2}(p)\tau_{1/2+ir_2}(n)- \tau_{1/2+ir_2}\left(\frac{n}{p}\right)\text{ if }(p,n)=p,$$
$$\tau_{1/2+ir_2}(np)= \tau_{1/2+ir_2}(p)\tau_{1/2+ir_2}(n)\text{ if }(p,n)=1.$$
This yields the result.
\end{proof}

\subsection{Applying the Poisson-type summation formula}

By proposition \ref{mainprop},  the term \eqref{eq:TMNC} can be decomposed as follows
\begin{equation}\label{eq:decomp}
T_{M,N}(c)=TS(c,0)-\tau_{1/2+ir_2}(p)TS(c,1)+TS(c,2),
\end{equation}
where
\begin{equation}\label{eq:TSCB}
TS(c,B)=c\sum_{m,n \geq 1}\tau_{1/2+ir_1}(m)\tau_{1/2+ir_2}(n)S(m,np^B,c)f(m,np^B,c)
\end{equation}
with $B=0,1,2$
and $f(m,n,c)$ is defined by \eqref{fmnc}.

\begin{prop} One has
\begin{equation}
TS(c,B)=TS^*(c,B)+TS^+(c,B)+TS^{-}(c,B),
\end{equation}
 where
\begin{equation}
TS^*(c,B)=\sum_{n \geq 1} \tau_{1/2+ir_2}(n)S(0,np^B,c)\left[G_{r_1}^{*}(np^B)+G_{-r_1}^{*}(np^B)\right],
\end{equation}
\begin{multline}
TS^{\mp}(c,B)=\sum_{m,n\geq 1} \tau_{1/2+ir_1}(m)\tau_{1/2+ir_2}(n)\\
\times S(0, np^B\mp m,c)G_{r_1}^{\mp}(m,np^B).
\end{multline}
The functions $G^{*}_{r}$, $G^{-}_{r}$, $G^{+}_{r}$ are defined as follows
\begin{equation}
G_{r}^{*}(y)=\frac{\zeta(1+2ir)}{c^{2ir}}\int_{0}^{\infty}J_{k-1}\left(\frac{4\pi\sqrt{xy}}{c}\right)F_{M,N}(x,y)x^{ir}dx,
\end{equation}
\begin{multline}\label{eq:G-}
G_{r}^{-}(z,y)=2\pi \int_{0}^{\infty} k_0\left(\frac{4\pi \sqrt{xz}}{c},1/2+ir\right)
\\ \times J_{k-1}\left(\frac{4\pi
\sqrt{xy}}{c}\right)F_{M,N}(x,y)dx,
\end{multline}
\begin{multline}\label{eq:G+}
G_{r}^{+}(z,y)=2\pi \int_{0}^{\infty} k_1\left(\frac{4\pi \sqrt{xz}}{c},1/2+ir\right)\\
\times J_{k-1}\left(\frac{4\pi
\sqrt{xy}}{c}\right)F_{M,N}(x,y)dx.
\end{multline}
\end{prop}
\begin{proof}
The function $f$ is smooth, compactly supported, and thus satisfies all conditions of theorem  \ref{kuzth}.
Applying the summation formula with $\phi(x):=f(\frac{c^2}{16\pi^2}x^2,np^B,c)$, we obtain
\begin{multline*}
\sum_{m \geq 1}e\left(\frac{md}{c}\right)\tau_{1/2+ir_1}(m)f(m,np^B,c)=\\ \frac{\zeta(1+2ir_1)}{c^{1+2ir_1}}\int_{0}^{\infty} f(x,np^B,c)x^{ir_1}dx+\frac{\zeta(1-2ir_1)}{c^{1-2ir_1}}\int_{0}^{\infty} f(x,np^B,c)x^{-ir_1}dx\\+\frac{2\pi}{c} \sum_{m \geq 1}\tau_{1/2+ir_1}(m)\int_{0}^{\infty}e\left(\frac{-m\overline{d}}{c}\right)k_0\left(\frac{4\pi}{c}\sqrt{xm},1/2+ir_1\right)f(x,np^B,c)dx\\+\frac{2\pi}{c} \sum_{m \geq 1}\tau_{1/2+ir_1}(m)\int_{0}^{\infty}e\left(\frac{m\overline{d}}{c}\right)k_1\left(\frac{4\pi}{c}\sqrt{xm},1/2+ir_1\right)f(x,np^B,c)dx.
\end{multline*}
Plugging this in \eqref{eq:TSCB} yields the assertion.
\end{proof}
The next lemma shows that $TS^*(c)$ term contributes to the fourth moment as an error.
\begin{lemma}
Let $l=0,1$. Then
\begin{equation}
\sum_{\frac{q}{p^l}|c}\sum_{M,N \leq q^{1+\epsilon}} c^{-2}TS^*(c,B)\ll_{\epsilon, \uple{r}} q^{-1+\epsilon}.
\end{equation}
\end{lemma}
\begin{proof}
We use lemma \ref{boundfmn} to estimate $F_{M,N}(m,n)$. The $J$-Bessel function can be trivially bounded by $1$. Then
$$G^{*}_{r}(np^B)\ll_{\uple{r}}\left(\frac{\hat{q}^2}{M}\right)^{|t_1|}\left(\frac{\hat{q}^2}{N}\right)^{|t_2|}\left(\frac{M}{N}\right)^{1/2}.$$ Since
$S(0,np^B,c) \ll(np^B, c),$ we have
$$ TS^*(c,B) \ll_{\uple{r}}  (MN)^{1/2}q^{\epsilon} \left(\frac{\hat{q}^2}{M}\right)^{|t_1|}\left(\frac{\hat{q}^2}{N}\right)^{|t_2|}.$$
Therefore,
\begin{equation*}
\sum_{\frac{q}{p^l}|c}\sum_{M,N \leq q^{1+\epsilon}} c^{-2}TS^*(c,B)\ll_{\epsilon, \uple{r}} q^{-1+\epsilon}.
\end{equation*}
\end{proof}

The last two summands require more detailed treatment.
We rewrite the sums $TS^{\pm}$ in the form that is more convenient for later computations
\begin{align*}TS^{-}(c,B)=\sum_{m\geq 1} \sum_{n\geq 1} \tau_{1/2+ir_1}(m)\tau_{1/2+ir_2}(n)S(0, np^B-
m,c)G_{r_1}^{-}(m,np^B) \\ =\phi(c)\sum_{n\geq 1}
\tau_{1/2+ir_1}(np^B)\tau_{1/2+ir_2}(n)G_{r_1}^{-}(np^B,np^B)+\sum_{h\neq 0}S(0,h,c)T_{h}^{-}(c,B)
\end{align*}
and
\begin{multline*} TS^{+}(c,B)=\sum_{m\geq 1} \sum_{n \geq 1} \tau_{1/2+ir_1}(m)\tau_{1/2+ir_2}(n)S(0, np^B+
m,c)G_{r_1}^{+}(m,np^B)\\ =\sum_{h\neq 0}S(0,h,c)T_{h}^{+}(c,B),
\end{multline*}
where
\begin{equation}\label{eq:ThcB} T_{h}^{\mp}(c,B)=\sum_{m\mp np^B=h}\tau_{1/2+ir_1}(m)\tau_{1/2+ir_2}(n)G_{r_1}^{\mp}(m,p^Bn).
\end{equation}
At this point, the non-diagonal term splits into off-diagonal (corresponds to $h=0$) and off-off-diagonal ($h \neq 0$) parts.
\begin{theorem} \label{split}
One has
\begin{equation}
M^{OD}=M^{OD}(0)-\tau_{1/2+ir_2}(p)M^{OD}(1)+M^{OD}(2),
\end{equation}
\begin{equation}\label{EOOD}
M^{OOD}=M^{OOD}(0)-\tau_{1/2+ir_2}(p)M^{OOD}(1)+M^{OOD}(2).
\end{equation}
For $B=0,1,2$
\begin{align*}
M^{OD}(B)=2\pi i^{-k}\biggl(\sum_{\substack{q|c \\ c \ll C}}\frac{\phi(c)}{c^2}
 \sum_{\substack{n\\M, N \ll q^{1+\epsilon}}} \tau_{1/2+ir_1}(np^B)\tau_{1/2+ir_2}(n)G_{r_1}^{-}(np^B,np^B)\\
-\frac{1}{p}\sum_{\substack{\frac{q}{p}|c\\ c \ll C}}\frac{\phi(c)}{c^2}\sum_{\substack{n\\M, N \ll q^{1+\epsilon}}} \tau_{1/2+ir_1}(np^B)\tau_{1/2+ir_2}(n)G_{r_1}^{-}(np^B,np^B)\biggr),
\end{align*}

\begin{align*}M^{OOD}(B)=2\pi i^{-k}\left( \sum_{\substack{q|c\\c \ll C}}\frac{1}{c^2}\sum_{M,N \ll q^{1+\epsilon}}\sum_{h \neq 0}S(0,h,c)(T_{h}^{-}(c,B)+T_{h}^{+}(c,B))\right.\nonumber\\
  \left.
-\frac{1}{p}\sum_{\substack{\frac{q}{p}|c\\ c \ll C}}\frac{1}{c^2}\sum_{M,N \ll q^{1+\epsilon}} \sum_{h \neq 0}S(0,h,c)(T_{h}^{-}(c,B)+T_{h}^{+}(c,B))\right).
\end{align*}
Here $T_{h}^{\pm}(c,B)$ is given by \eqref{eq:ThcB} and $G_{r}^{\pm}(z,y)$ by \eqref{eq:G-}, \eqref{eq:G+}.
\end{theorem}

\section{Asymptotic evaluation of the diagonal and off-diagonal terms}\label{asympD}
The main result of this section is the asymptotic formula for the diagonal and off-diagonal terms.
\begin{theorem} Up to a negligible error term, we have
\begin{multline}
M^D+M^{OD}=\frac{\phi(q)}{q} \sum_{\epsilon_1, \epsilon_2=\pm 1}\hat{q}^{-2t_1-2t_2+2\epsilon_1t_1+2\epsilon_2t_2}
\frac{\Gamma(\epsilon_1t_1+ir_1+k/2)}{\Gamma(t_1+ir_1+k/2)}
\\
\times \frac{\Gamma(\epsilon_1t_1-ir_1+k/2)}{\Gamma(t_1-ir_1+k/2)}   \frac{\Gamma(\epsilon_2t_2+ir_2+k/2)\Gamma(\epsilon_2t_2-ir_2+k/2)}{\Gamma(t_2+ir_2+k/2)\Gamma(t_2-ir_2+k/2)}
\\
\times \zeta_q(1+2\epsilon_1t_1) \zeta_q(1+2\epsilon_2t_2)\frac{\prod_{\epsilon_3, \epsilon_4=\pm 1} \zeta_q(1+\epsilon_1t_1+\epsilon_2t_2+ i\epsilon_3r_1+ i\epsilon_4r_2) }{\zeta_q(2+2\epsilon_1t_1+2\epsilon_2t_2)}.
\end{multline}
\end{theorem}
\subsection{Extension of summations}
First, we reintroduce the summation over $c>C$ and $\max{(M,N)}\gg q^{1+\epsilon} $ for the off-diagonal term at the cost of admissible error.
\begin{prop}
For any $\epsilon>0$
\begin{multline}
\sum_{\substack{\frac{q}{p^l}|c\\ c > C }}\frac{\phi(c)}{c^2}\sum_{\max{(M,N)}\ll q^{1+\epsilon} }\sum_{n} \tau_{1/2+ir_1}(np^B)\tau_{1/2+ir_2}(n)G_{r_1}^{-}(np^B,np^B)
\\\ll_{\epsilon, \uple{r}}   q^{\epsilon-\frac{k-1}{8-8\theta}} .
\end{multline}
\end{prop}
\begin{proof}
Let
\begin{equation*}
\eta_C(c)=\begin{cases}
0 &\mbox{if } c>C \\
1 & \mbox{if } c \leq C.
\end{cases}
\end{equation*}
Consider
\begin{multline*}
T_1:=\sum_{\substack{\frac{q}{p^l}|c\\ c > C }}\frac{\phi(c)}{c^2}\sum_{\max{(M,N)}\ll q^{1+\epsilon} }\sum_{n} \tau_{1/2+ir_1}(np^B)\tau_{1/2+ir_2}(n)G_{r_1}^{-}(np^B,np^B)\\
= \sum_{\max{(M,N)}\ll q^{1+\epsilon} }\sum_{n} \tau_{1/2+ir_1}(np^B)\tau_{1/2+ir_2}(n) \int_{0}^{\infty}k_0\left( 4\pi\sqrt{xnp^{B}},1/2+ir_1\right)\\
\times 2\pi J_{k-1}\left(4\pi\sqrt{xnp^{B}}\right)\sum_{\frac{q}{p^l}|c}(1-\eta_{C}(c))\phi(c)F_{M,N}(xc^2,np^B)dx.
\end{multline*}
We use lemma \ref{boundfmn} to bound $F_{M,N}(xc^2,np^B)$, formula \eqref{eq:Jbes} to bound the Bessel function
$J_{k-1}\left(4\pi\sqrt{xnp^{B}}\right)$ and trivial estimate for the Bessel kernel
$$k_0\left( 4\pi\sqrt{xnp^{B}},1/2+ir_1\right)\ll 1.$$
Then
\begin{equation*}
T_1 \ll_{\epsilon, \uple{r}}q^{\epsilon}\sum_{M,N \ll q^{1+\epsilon}}\frac{1}{\sqrt{MN}}\sum_{n \sim N}\int_{0}^{2M/C^2}(\sqrt{xn})^{k-1}\frac{M}{qx}dx
\ll_{\epsilon, \uple{r}} q^{\epsilon-\frac{k-1}{8-8\theta}}.
\end{equation*}
\end{proof}
\begin{prop}
For any $\epsilon>0$, any $A>0$ and  $l=0,1$
\begin{multline}
\sum_{\substack{\frac{q}{p^l}|c }}\frac{\phi(c)}{c^2}\sum_{\max{(M,N)}\gg q^{1+\epsilon} }\sum_{n} \tau_{1/2+ir_1}(np^B)\tau_{1/2+ir_2}(n)G_{r_1}^{-}(np^B,np^B)\\ \ll_{\epsilon, A, \uple{r}} q^{-A}.
\end{multline}
\end{prop}
\begin{proof}
The statement can be proved analogously to proposition \ref{prop:A}.
\end{proof}
Now  it is possible to combine all functions $F_M$ into $F$ and replace $\sum_{M,N} F_{M,N}$ by
\begin{equation}\label{fxy}
F(x,y):=\frac{1}{\sqrt{xy}}W_{t_1,r_1}\left(\frac{x}{\hat{q}^2}\right)W_{t_2,r_2}\left(\frac{y}{\hat{q}^2}\right)F(x)F(y),
\end{equation}
where $F(x)$ is a smooth function, compactly supported in $[1/2, \infty)$ such that $F(x)=1$ for $x \geq 1$.

\begin{prop}
Up to the error term $O_{\uple{r}, \epsilon}( q^{\epsilon-k/2}),$ the product $F(x)F(y)$ can be replaced by $1$ in \eqref{fxy}.
\end{prop}

\begin{proof}
Consider
\begin{multline*}
T_2:=\sum_{\substack{c\\ q|c }}\frac{\phi(c)}{c^2}\sum_{n} \tau_{1/2+ir_1}(np^B)\tau_{1/2+ir_2}(n)
\int_{0}^{1}k_0\left( \frac{4\pi\sqrt{xnp^{B}}}{c},1/2+ir_1\right)\\
\times J_{k-1}\left(\frac{4\pi\sqrt{xnp^{B}}}{c}\right)\frac{1}{\sqrt{xnp^B}}W_{t_1,r_1}\left(\frac{x}{\hat{q}^2}\right)W_{t_2,r_2}\left(\frac{np^B}{\hat{q}^2}\right)(1-F(x))dx.
\end{multline*}
We estimate the kernel $k_0\left( \frac{4\pi\sqrt{xnp^{B}}}{c},1/2+ir_1\right)$ trivially by 1 and apply the following bound for the $J$-Bessel function
$$J_{k-1} \left(\frac{4\pi\sqrt{xnp^{B}}}{c}\right) \ll \left(\frac{\sqrt{xn}}{c}\right)^{k-1}.$$
If $n<q$, the function
$W_{t_2,r_2}$ can be estimated using \eqref{eq:W0bound}. Otherwise
we apply  \eqref{eq:Wbound}. This gives
$$T_2 \ll_{\uple{r}} q^{\epsilon-k/2}.$$

\end{proof}

\subsection{Asymptotics of diagonal and off-diagonal terms} \label{off-diagonal}
The off-diagonal term can be written as
\begin{equation*}
M^{OD}(B)=\hat{q}^{-2t_1-2t_2}\sum_n\frac{\tau_{1/2+ir_1}(np^B)\tau_{1/2+ir_2}(n)}{np^B}
W_{t_2,r_2}\left(\frac{np^B}{\hat{q}^2}\right)Z(np^B)
\end{equation*}
for $B=0,1,2$ with
\begin{multline*}
Z(u):=2\pi i^{-k}\int_{0}^{\infty}k_0(z,1/2+ir_1)J_{k-1}(z)\\
\times \left(\sum_{q|c}\frac{\phi(c)}{c}W_{t_1,r_1}\left(\frac{z^2c^2}{(4\pi)^2
\hat{q}^2 u}\right)-
\frac{1}{p}\sum_{\frac{q}{p}|c}\frac{\phi(c)}{c}W_{t_1,r_1}\left(\frac{z^2c^2}{(4\pi)^2\hat{q}^2u}\right) \right)dz.
\end{multline*}
Note that we made the change of variables
$x=\frac{z^2c^2}{(4\pi)^2u}$ in the integral.
Applying \eqref{eq:W}, we have
\begin{multline*}
Z(u)=2 \pi i^{-k}\int_{0}^{\infty}k_0(z,1/2+ir_1)J_{k-1}(z)\\
\times \frac{1}{2\pi
i}\int_{(3)}P_r(s)\frac{\Gamma(s+ir_1+k/2)\Gamma(s-ir_1+k/2)}{\Gamma(t_1+ir_1+k/2)\Gamma(t_1-ir_1+k/2)} \\
\times\zeta_q(1+2s) \left(\frac{z^2}{(4\pi)^2\hat{q}^2u}\right)^{-s}\times \left[
\sum_{q|c}\frac{\phi(c)}{c^{1+2s}}-
\frac{1}{p}\sum_{\frac{q}{p}|c } \frac{\phi(c)}{c^{1+2s}}
\right]\frac{2sds}{s^2-t_{1}^{2}}dz.
\end{multline*}
The term in the brackets can be simplified
\begin{equation*}\sum_{q|c}\frac{\phi(c)}{c^{1+2s}}-
\frac{1}{p}\sum_{\frac{q}{p}|c}\frac{\phi(c)}{c^{1+2s}}=\frac{\phi(q)}{q^{1+2s}}\frac{1-p^{2s-1}}{1-p^{-2s}}\frac{\zeta_q(2s)}{\zeta_q(2s+1)}.
\end{equation*}
Lemma \ref{BesselInt} implies that
\begin{multline*}
\int_{0}^{\infty}k_0(z,1/2+ir_1)J_{k-1}(z)z^{-2s}dz=\frac{\Gamma(2s)}{2^{2s+1}\cos{(\pi(1/2+ i r_1))}}\\
\times \Biggl(\frac{\Gamma(ir_1+k/2-s)}{\Gamma(-ir_1+k/2+s)\Gamma(ir_1+k/2+s)\Gamma(ir_1-k/2+s+1)}
\\-\frac{\Gamma(-ir_1+k/2-s)}{\Gamma(-ir_1+k/2+s)\Gamma(ir_1+k/2+s)\Gamma(-ir_1-k/2+s+1)}
\Biggr).
\end{multline*}
By duplication  and reflection formulas
\begin{multline*}
\int_{0}^{\infty}k_0(z,1/2+ir_1)J_{k-1}(z)z^{-2s}dz=-\frac{i^k\Gamma(s)\Gamma(s+1/2)}{2^{2}\pi^{3/2}\sin{(\pi i r_1)}}\\
\times \frac{\Gamma(ir_1+k/2-s)\Gamma(-ir_1+k/2-s)}{\Gamma(ir_1+k/2+s)\Gamma(-ir_1+k/2+s)}[\sin{\pi(-s-ir_1)}-\sin{\pi(-s+ir_1)}].
\end{multline*}
Observe that
\begin{equation*}
\frac{\Gamma(1/2-s)\Gamma(1/2+s)}{2\pi\sin{(\pi i r_1)}}[\sin{(\pi(-s-ir_1))}-\sin{(\pi(-s+ir_1))}]=-1.
\end{equation*}
Consequently,
\begin{multline*}Z(u)=\frac{\phi(q)}{ q} \frac{1}{2\pi i}\int_{(3)}P_r(s)
\frac{\Gamma(-s+ir_1+k/2)\Gamma(-s-ir_1+k/2)}{\Gamma(t_1+ir_1+k/2)\Gamma(t_1-ir_1+k/2)}\\
\times \zeta_q(1-2s)\left(\frac{u}{\hat{q}^2}\right)^s\frac{2sds}{s^2-t_{1}^{2}}.
\end{multline*}
Shifting the contour of integration to $\Re (s)=-3$, we cross
poles at $s=\pm t_1$.
Hence
\begin{multline*}
Z(u)=\frac{\phi(q)}{ q} \sum_{\epsilon_1=\pm 1}\frac{\Gamma(\epsilon_1t_1+ir_1+k/2)\Gamma(\epsilon_1t_1-ir_1+k/2)}{\Gamma(t_1+ir_1+k/2)\Gamma(t_1-ir_1+k/2)} \\
\times \zeta_q(1+2\epsilon_1t_1)\left(\frac{u}{\hat{q}^2}\right)^{-\epsilon_1t_1}
-\frac{\phi(q)}{ q} \frac{1}{2\pi i}\int_{(3)}P_r(s)\\
\times
\frac{\Gamma(s+ir_1+k/2)\Gamma(s-ir_1+k/2)}{\Gamma(t_1+ir_1+k/2)\Gamma(t_1-ir_1+k/2)}\zeta_q(1+2s)\left(\frac{u}{\hat{q}^2}\right)^{-s}\frac{2sds}{s^2-t_{1}^{2}}.
\end{multline*}
Substitution of $Z(np^B)$ into $M^{OD}(B)$ gives
\begin{multline*}
M^{OD}(B)=\frac{\phi(q)}{ q}\hat{q}^{-2t_1-2t_2}\sum_n\frac{\tau_{1/2+ir_1}(np^B)\tau_{1/2+ir_2}(n)}{np^B}
W_{t_2,r_2}(\frac{np^B}{\hat{q}^2})\\
\times \Biggl( -W_{t_1,r_1}\left(\frac{
np^B}{\hat{q}^2}\right) + \sum_{\epsilon_1=\pm 1} \frac{\Gamma(\epsilon_1t_1+ir_1+k/2)\Gamma(\epsilon_1t_1-ir_1+k/2)}{\Gamma(t_1+ir_1+k/2)\Gamma(t_1-ir_1+k/2)} \\  \times \zeta_q(1+2\epsilon_1t_1)\left(\frac{np^B}{\hat{q}^2}\right)^{-\epsilon_1t_1} \Biggr).
\end{multline*}
Property of multiplicity \eqref{multtau} implies that
\begin{multline*}
\sum_{\substack{n \geq 1 \\(n,p)=1}}\tau_{1/2+ir_2} (n)f(n)=\sum_{n\geq 1}\tau_{1/2+ir_2} (n)f(n)\\ -\tau_{1/2+ir_2} (p)\sum_{n\geq 1}\tau_{1/2+ir_2}(n)f(np)+\sum_{n\geq 1}\tau_{1/2+ir_2} (n)f(np^2).
\end{multline*}
Thus,
\begin{multline}\label{shiftedE}
M^{D}+M^{OD}=\frac{\phi(q)}{q}\hat{q}^{-2t_1-2t_2}
\sum_{(n,p)=1}\frac{\tau_{1/2+ir_1}(n)\tau_{1/2+ir_2}(n)}{n}
W_{t_2,r_2}(\frac{n}{\hat{q}^2}) \\ \times \left( \sum_{\epsilon_1=\pm 1}\frac{\Gamma(\epsilon_1t_1+ir_1+k/2)\Gamma(\epsilon_1t_1-ir_1+k/2)}{\Gamma(t_1+ir_1+k/2)\Gamma(t_1-ir_1+k/2)} \zeta_q(1+2\epsilon_1t_1)\left(\frac{n}{\hat{q}^2}\right)^{-\epsilon_1t_1}      \right).
\end{multline}
Ramanujan's identity \eqref{ramshift} yields
\begin{equation*}
\sum_{(n,p)=1}\frac{\tau_{1/2+ir_1}(n)\tau_{1/2+ir_2}(n)}{n^{1+\epsilon_1t_1+s}}=\frac{ \prod_{\epsilon_3, \epsilon_4=\pm 1}
\zeta_q(1+\epsilon_1t_1+s+i\epsilon_3r_1+ i\epsilon_4r_2)}{\zeta_q(2+2\epsilon_1t_1+2s)}.
\end{equation*}
Therefore,
\begin{multline*}
M^D+M^{OD}=\frac{\phi(q)}{q} \sum_{\epsilon_1=\pm 1}\frac{\Gamma(\epsilon_1t_1+ir_1+k/2)\Gamma(\epsilon_1t_1-ir_1+k/2)}{\Gamma(t_1+ir_1+k/2)\Gamma(t_1-ir_1+k/2)}     \\
\times \hat{q}^{-2t_1-2t_2+2\epsilon_1t_1}
\frac{1}{2\pi i}\int_{\Re{s}=3}\frac{P_r(s)}{P_r(t_2)}\hat{q}^{2s}\zeta_q(1+2s)\zeta_q(1+2\epsilon_1t_1)\\
\times \frac{\Gamma(s+ir_2+k/2)\Gamma(s-ir_2+k/2)}{\Gamma(t_2+ir_2+k/2)\Gamma(t_2-ir_2+k/2)}\\
\times
\frac{\prod_{\epsilon_3, \epsilon_4=\pm 1} \zeta_q(1+\epsilon_1t_1+s+ i\epsilon_3r_1+ i\epsilon_4r_2) }{\zeta_q(2+2\epsilon_1t_1+2s)}\frac{2sds}{s^2-t_{2}^{2}}.
\end{multline*}
Shifting the contour of integration to $\Re{s}=-1/2$, the resulting integral is bounded by $q^{\epsilon-1/2}$ plus the contribution of simple poles at $s=\pm t_2$.
Up to an error term,
\begin{multline*}
M^D+M^{OD}=\frac{\phi(q)}{q} \sum_{\epsilon_1, \epsilon_2=\pm 1}\hat{q}^{-2t_1-2t_2+2\epsilon_1t_1+2\epsilon_2t_2} \frac{\Gamma(\epsilon_1t_1+ir_1+k/2)}{\Gamma(t_1+ir_1+k/2)}\\
\times \frac{\Gamma(\epsilon_1t_1-ir_1+k/2)}{\Gamma(t_1-ir_1+k/2)}   \frac{\Gamma(\epsilon_2t_2+ir_2+k/2)\Gamma(\epsilon_2t_2-ir_2+k/2)}{\Gamma(t_2+ir_2+k/2)\Gamma(t_2-ir_2+k/2)} \\
\times \zeta_q(1+2\epsilon_1t_1) \zeta_q(1+2\epsilon_2t_2)
\frac{\prod_{\epsilon_3, \epsilon_4=\pm 1} \zeta_q(1+\epsilon_1t_1+\epsilon_2t_2+ i\epsilon_3r_1+ i\epsilon_4r_2) }{\zeta_q(2+2\epsilon_1t_1+2\epsilon_2t_2)}.
\end{multline*}
By letting shifts tend to zero in \eqref{shiftedE}, we find
\begin{equation}
M^{D}+M^{OD}=\left(\frac{\phi(q)}{q}\right)^2
\sum_{(n,p)=1}\frac{\tau(n)^2}{n}
W_{0,0}(\frac{n}{\hat{q}^2})\log{\left(\frac{\hat{q}^2}{n}\right)}.
\end{equation}
The equality \eqref{Ramid} gives
\begin{multline*}
M^{D}+M^{OD}=\frac{1}{2\pi i}\left(\frac{\phi(q)}{q}\right)^2\int_{(3)}\frac{P_r(s)}{P_r(0)}\frac{\Gamma(k/2+s)^2}{\Gamma(k/2)^2}\zeta_q(1+2s)\hat{q}^{2s}\frac{\zeta_q(1+s)^4}{\zeta_q(2+2s)}\\ \times \left[\log{\hat{q}^2}+4\frac{\zeta_{q}^{'}}{\zeta_q}(1+s)-2\frac{\zeta_{q}^{'}}{\zeta_q}(2+2s)\right]\frac{2ds}{s}.
\end{multline*}
Shifting the contour of integration to $\Re{s}=-1/2$, the resulting integral is bounded by $q^{-1/2}$ plus the contribution of multiple poles at $s=0$. Calculation of the residue
$$\left( \frac{\phi(q)}{q}\right)^7\frac{1}{\zeta_q(2)}\res_{s=0}\frac{\hat{q}^{2s}}{s^6}\left(\log{\hat{q}}-\frac{4}{s}\right)$$
shows that the main term is $$ \left(\frac{\phi(q)}{q}\right)^7 \frac{p^2}{p^2-1}\frac{(\log{q})^6}{60\pi^2}.$$

\section{ Off-off-diagonal term: double integral representation}\label{explicitformula}
In this section, we will show that the off-off-diagonal main term can be written as a double integral.
\begin{theorem}\label{mainOOD}
Up to a negligible error, we have
\begin{multline}
M^{OOD}=\frac{\phi(q)}{q}\sum_{\epsilon_1,\epsilon_2=\pm 1}\frac{\zeta_q(1+2i\epsilon_1r_1)\zeta_q(1+2i\epsilon_2r_2)}{\zeta_q(2+2i\epsilon_1r_1+2i\epsilon_2r_2)}\hat{q}^{-2t_1-2t_2}\\
\times \hat{q}^{-2i\epsilon_1r_1+2i\epsilon_2r_2}\frac{1}{(2\pi i)^2}\int_{\Re{t}=k/2+0.7}\int_{\Re{s}=k/2-0.4}I_{\epsilon_1,\epsilon_2}(s,t)\frac{2sds}{s^2-t_{1}^{2}}\frac{2tdt}{t^2-t_{2}^{2}},
\end{multline}
where
\begin{multline}
I_{\epsilon_1,\epsilon_2}(s,t)=
\frac{P_r(s)P_r(t)}{P_r(t_1)P_r(t_2)}\zeta_q(1+t+s+i\epsilon_1r_1+i\epsilon_2r_2)\\
\times \zeta_q(1+t-s+i\epsilon_1r_1+i\epsilon_2r_2) \zeta_q(1-t+s+i\epsilon_1r_1+i\epsilon_2r_2) \zeta_q(1-t-s+i\epsilon_1r_1+i\epsilon_2r_2)\\
\times \frac{\Gamma(k/2+s +i\epsilon_1r_1)\Gamma(k/2+t +i\epsilon_2r_2)}{\Gamma(k/2+t_1+ir_1)\Gamma(k/2+t_1-ir_1)} \frac{\Gamma(k/2-s +i\epsilon_1r_1)\Gamma(k/2-t +i\epsilon_2r_2)}{\Gamma(k/2+t_2+ir_2)\Gamma(k/2+t_2-ir_2)}.
\end{multline}
\end{theorem}
\subsection{Estimation of $G^{\mp}_{r_1}$}
The expression $$T_{h}^{\pm}(c,B)=\sum_{m\pm np^B=h}\tau_{1/2+ir_1}(m)\tau_{1/2+ir_2}(n)G^{\pm}_{r_1}(m,p^Bn)$$ can be evaluated using theorem \ref{thm:DFI1}. To this end, we show that the functions $G^{\pm}_{r_1}$, defined by \eqref{eq:G-} and \eqref{eq:G+}, satisfy condition \eqref{eq:condition1}.

Let  $Q:= 1+\frac{\sqrt{MN}}{c}$, $Z:=\frac{Q^2c^2}{M}$, $Y:=N$.
\begin{lemma}
For all positive $n_1$ and $n_2$
\begin{multline}
z^j y^i\frac{\partial^j}{\partial z^j}\frac{\partial^i}{\partial y^i}G^{\pm}_{r_1}(z,y)\ll  (1+\frac{z}{Z})^{-n_1}(1+\frac{y}{Y})^{-n_2} \frac{M^{1/2}}{N^{1/2}}\\
\times \left( \frac{\sqrt{MN}}{c}\right)^{k-1}Q^{j+i-k+1/2}.
\end{multline}
\end{lemma}
\begin{proof}
Consider
\begin{multline*}
G_{r_1}^{-}(z,y)= 2\pi \int_{0}^{\infty} k_0\left(\frac{4\pi \sqrt{xz}}{c},1/2+ir_1\right)J_{k-1}\left(\frac{4\pi
\sqrt{xy}}{c}\right)\\ \times F_{M,N}(x,y)dx
=-\frac{\pi}{\sin{\pi ir_1}} \int_{0}^{\infty} \left[J_{2ir_1}\left(\frac{4\pi \sqrt{xz}}{c}\right)-J_{-2ir_1}\left(\frac{4\pi \sqrt{xz}}{c}\right)\right]\\
\times J_{k-1}\left(\frac{4\pi
\sqrt{xy}}{c}\right)F_{M,N}(x,y)dx.
\end{multline*}
Suppose that $z > Z$.
Let $u:=\frac{4\pi\sqrt{xz}}{c}$. Then
\begin{multline*}
G_{r_1}^{-}(z,y)=-\frac{c^2}{8 \pi z\sin{\pi i r_1}} \int_{0}^{\infty} u \left[J_{2ir_1}(u)-J_{-2ir_1}(u)\right]\\
\times J_{k-1}\left( u \sqrt{\frac{y}{z}}
\right)F_{M,N}(\frac{c^2 u^2}{16 \pi^2 z},y) du.
\end{multline*}
It is sufficient to estimate
\begin{multline*}G_1(z,y):=-\frac{c^2}{8 \pi z\sin{\pi i r_1}} \int_{0}^{\infty} u J_{2ir_1}(u)J_{k-1}\left( u \sqrt{\frac{y}{z}}
\right)\\
\times F_{M,N}\left(\frac{c^2 u^2}{16 \pi^2 z},y\right) du.
\end{multline*}
Note that $F_{M,N}(x,y)$ is compactly supported on $[M/2,3M]\times [N/2,3N]$.
Let $$f(u):=g_1(u)g_2(u)u^{-2ir_1}$$ with $$g_1(u):=J_{k-1}\left( u \sqrt{\frac{y}{z}} \right) \text{ and } g_2(u):=F_{M,N}\left(\frac{c^2 u^2}{16
\pi^2 z},y\right).$$
The recurrent relation \eqref{recBess} implies that
$$G_{1}(z,y)=-\frac{c^2}{8 \pi z\sin{\pi i r_1}}  \int_{0}^{\infty} \left(u^{1+2ir_1}J_{1+2ir_1}(u)\right){'}f(u)du.$$
Integration by parts gives
\begin{multline*}G_{1}(z,y)=\frac{c^2}{8 \pi z\sin{\pi i r_1}}  \int_{0}^{\infty} u^{1+2ir_1}J_{1+2ir_1}(u)f'(u)du\\=
-\frac{c^2}{8 \pi z\sin{\pi i r_1}}
\int_{0}^{\infty} u^{2+2ir_1}J_{2+2ir_1}(u)(\frac{1}{u}f'(u))'du.
\end{multline*}
Repeating the procedure $n$ times, we have
\begin{multline*}
G_{1}(z,y)=(-1)^{n+1}\frac{c^2}{8 \pi z\sin{\pi i r_1}} \int_{0}^{\infty}
u^{n+2ir_1}J_{n+2ir_1}(u)h_n(u)du=\\
(-1)^{n+1}\frac{c^2}{8 \pi z\sin{\pi i r_1}}\int_{u \sim \frac{\sqrt{Mz}}{c}} \frac{J_{n+2ir_1}(u)}{u^{n-1-2ir_1}} u^{2n-1}h_n(u)du,
\end{multline*}
where
$$h_0(u)=f(u),$$
$$h_1(u)=f'(u),$$
$$h_n(u)=(u^{-1}h_{n-1}(u))'\text{ for }n \geq 2.$$
By induction for $n \geq 1$
$$u^{2n-1}h_n(u)=\sum_{i=0}^{n}c(i,n)f^{(i)}(u)u^i$$
with
\begin{multline*}f^{(i)}(u)u^i\ll \sum_{j+l+m=i}(g_{1}^{(j)}(u)
u^j)(g_{2}^{(l)}(u)u^{l})u^{2ir_1} \\ \ll \sum_{j+m<i}(g_{1}^{(j)}(u)
u^j)(g_{2}^{(i-j-m)}(u)u^{i-j-m}).
\end{multline*}
Fa\'{a} di Bruno's formula and the estimate \eqref{eq:Jbes} give
\begin{multline*}u^jg_{1}^{(j)}(u) =u^j\frac{\partial^j}{\partial u^j}\left(J_{k-1}\left( u
\sqrt{\frac{y}{z}} \right)\right)= \left(\sqrt{\frac{y}{z}}u\right)^j J_{k-1}^{(j)}\left( u
\sqrt{\frac{y}{z}} \right)\\ \ll \frac{ \left(u \sqrt{\frac{y}{z}}\right)^{k-1 }(1+u\sqrt{\frac{y}{z}})^j}{(1+u \sqrt{\frac{y}{z}})^{k-1/2}}.
\end{multline*}
Applying Fa\'{a} di Bruno's formula to the second function, we obtain
\begin{multline*}u^{i-j-m}g_{2}^{(i-j-m)}(u)=u^{i-j-m}\frac{\partial^{i-j-m}}{\partial
u^{i-j-m}}F_{M,N}\left(\frac{c^2 u^2}{16 \pi^2 z},y\right)\\=u^{i-j-m}\sum_{\substack{(m_1,m_2)\\m_1+2m_2=i-j-m}} \frac{(i-j-m)!}{m_1!m_2!(2!)^{m_2}}\\
\times  F^{(m_1+m_2)}_{M,N}\left(\frac{c^2 u^2}{16 \pi^2 z},y\right)
\left(\frac{2c^2}{16\pi^2z}u\right)^{m_1}\left(\frac{2c^2}{16\pi^2z} \right)^{m_2}.
\end{multline*}
Lemma \ref{boundfmn} implies
\begin{multline*}
u^{i-j-m}g_{2}^{(i-j-m)}(u)
\ll \sum_{\substack{(m_1,m_2)\\m_1+2m_2=i-j-m}} \left(\frac{c^2}{16 \pi^2z}u^2\right)^{m_1+m_2}\\
\times F_{M,N}^{(m_1+m_2)}\left(\frac{c^2 u^2}{16 \pi^2 z},y\right)  \ll (MN)^{-1/2}.
\end{multline*}
And for the $J$-Bessel function we use the trivial bound
$J_{n+2ir_1}(u) \ll 1.$
Then
$$G_{1}(z,y)\ll  \left(\frac{Qc}{\sqrt{Mz}}\right)^n
\frac{M^{1/2}}{N^{1/2}}\left(\frac{\sqrt{MN}}{c} \right)^{k-1}Q^{-k+1/2}$$ for every integer
$n>0$. The same bound is valid for $G_{r_1}^{-}(z,y)$.
So, if $z>Z$, the value of $G_{r_1}^{-}(z,y)$ is small.

Suppose $z \leq Z$. One can estimate $G_{r_1}^{-}(z,y)$ directly (without
integration by parts)
\begin{equation*}G_{r_1}^{-}(z,y) \ll  \frac{M^{1/2}}{N^{1/2}}\left( \frac{\sqrt{MN}}{c}\right)^{k-1}Q^{-k+1/2}.
\end{equation*}
Since $y \in [N/2,3N]$, we can add a multiple $(1+\frac{y}{Y})^{-n_2}$.

Combining two estimates for $G_{r_1}^{-}(z,y)$ in one, we have that for all positive $n_1$ and $n_2$
$$G_{r_1}^{-}(z,y) \ll (1+\frac{z}{Z})^{-n_1}(1+\frac{y}{Y})^{-n_2}\frac{M^{1/2}}{N^{1/2}}\left( \frac{\sqrt{MN}}{c}\right)^{k-1}Q^{-k+1/2}.$$
Analogously, using  relation \eqref{RecK} and  bound for the $K$-Bessel function \eqref{BESSKJ}, we may estimate $G^{+}_{r_1}(z,y)$.
Finally, differentiating $G_{r_1}^{\pm}(z,y)$ in $z$ variable $j$ times and in $y$ variable $i$ times, we find
\begin{multline*}
z^j y^i \frac{\partial^j}{\partial z^j}\frac{\partial^i}{\partial y^i}G_{r_1}^{\pm}(z,y)\ll
(1+\frac{z}{Z})^{-n_1}(1+\frac{y}{Y})^{-n_2}\\
\times \frac{M^{1/2}}{N^{1/2}}\left( \frac{\sqrt{MN}}{c}\right)^{k-1}Q^{j+i-k+1/2}
\end{multline*}
for all positive $n_1$ and $n_2$. An extra multiple of $Q^{i+j}$ is obtained by differentiating  the Bessel functions under the integral.
Indeed, by Fa\'{a} di Bruno's formula
\begin{multline*}
z^j\frac{\partial^j}{\partial z^j}J_{2ir_1}(\alpha\sqrt{z})=z^j
\sum \binom{j}{m_1,m_2, \ldots,m_j}\\
\times J_{2ir_1}^{(m_1+m_2+\ldots+m_j)}(\alpha\sqrt{z})\cdot \prod_{n=1}^{j}\left(\frac{\alpha z^{1/2-n}}{n!}\right)^{m_n},
\end{multline*}
where $$\alpha:=\frac{4\pi
\sqrt{x}}{c}$$ and the sum is over all $j$-tuples $(m_1, m_2, \ldots,m_j)$ such that
$1\cdot m_1+2\cdot m_2+ \ldots j\cdot m_j=j.$
Formula \eqref{jbesder} gives
\begin{equation*}J_{2ir_1}^{(b)}(z)=\\ \frac{1}{2^{b}}\sum_{t=0}^{b}(-1)^t \binom{b}{t}J_{2ir_1-b+2t}(z).
\end{equation*}
When $z>Z$, the maximum of $z^j\frac{\partial^j}{\partial z^j}J_{2ir_1}(\alpha\sqrt{z})$ is attained when $m_1+m_2+\ldots+m_j=j$. Therefore,
$$z^j\frac{\partial^j}{\partial z^j}J_{2ir_1}(\alpha\sqrt{z}) \ll (\alpha\sqrt{z})^{j} \sum_{t=0}^{j}J_{2ir_1-j+2t}(\alpha\sqrt{z}).$$
This gives an extra multiple $\left( \frac{\sqrt{Mz}}{c}\right)^j$  and
$$G_{r_1}^{-}(z,y)\ll \left(\frac{Qc}{\sqrt{Mz}}\right)^{n-j}
Q^{j}\frac{M^{1/2}}{N^{1/2}}\left(\frac{\sqrt{MN}}{c} \right)^{k-1}Q^{-k+1/2}$$ for every integer
$n>0$.

In the similar manner
\begin{equation*}
y^i\frac{\partial^i}{\partial y^{i}}\left[J_{k-1}\left(\alpha \sqrt{y}\right)F_{M,N}(x,y) \right]\ll
\sum_{a=0}^{i}y^{a}\left(J_{k-1}(\alpha \sqrt{y}) \right)^{(a)}y^{i-a}F_{M,N}(x,y)^{(i-a)}
\end{equation*}
gives an extra factor of $Q^i$.
\end{proof}

\subsection{Applying theorem \ref{thm:DFI1}}

According to the formula \eqref{EOOD}, the off-off-diagonal term is equal to
\begin{equation}
M^{OOD}=M^{OOD}(0)-\tau_{1/2+ir_2}(p)M^{OOD}(1)+M^{OOD}(2),
\end{equation}
where for $B=0,1,2$
\begin{align*}
M^{OOD}(B)=\frac{2\pi i^{-k} }{\hat{q}^{2t_1+2t_2}}\left( \sum_{M,N \ll q^{1+\epsilon}}\sum_{\substack{q|c\\c \ll C}}\frac{1}{c^2}\sum_{h \neq 0}S(0,h,c)(T_{h}^{-}(c,B)+T_{h}^{+}(c,B))\right.\nonumber \\ \left.
-\frac{1}{p} \sum_{M,N \ll q^{1+\epsilon}}\sum_{\substack{\frac{q}{p}|c\\c \ll C}}\frac{1}{c^2}\sum_{h \neq 0}S(0,h,c)(T_{h}^{-}(c,B)+T_{h}^{+}(c,B)) \right)
.
\end{align*}
Since $k$ is even, $i^{-k}=i^k$.
\begin{lemma}\label{choice} Up to the error $$O_{ \uple{r},\epsilon}\left(q^{\epsilon}\left(q^{-\frac{k-1-2\theta}{8-8\theta}}+q^{-1/4}\right)\right),$$ we have
\begin{equation}\label{intOOD}
T_{h}^{\mp}(c,B)=\pm\int_{0}^{\infty}\delta_{h\pm p^By>0}G^{\mp}_{r_1}(h\pm
p^By,p^By)\Lambda(h\pm p^By,p^By)dy
\end{equation}
with
\begin{multline}\label{intOOD2}
\Lambda(h\pm p^By,p^By):=\sum_{w=1}^{\infty}S(0,h,w)\sum_{\epsilon_1,\epsilon_2= \pm 1}\frac{(p^B,w)^{1+2i\epsilon_2r_2}}{w^{2+2i\epsilon_1r_1+2i\epsilon_2r_2}}\\ \times \zeta(1+2i\epsilon_1r_1)
\zeta(1+2i\epsilon_2r_2)
(h\pm p^By)^{i\epsilon_1r_1}y^{i\epsilon_2r_2}.
\end{multline}

\end{lemma}
\begin{proof}
We apply theorem \ref{thm:DFI1} to the function $T_{h}^{\mp}(c,B)$ and let $x=h\pm p^By$. Then
$$T_{h}^{\mp}(c,B)=\pm \int_{0}^{\infty}\delta_{h\pm p^By>0}G^{\mp}_{r_1}(h\pm
p^By,p^By)\Lambda(h\pm p^By,p^By)
dy+O(ET),$$
 where
\begin{multline*}
\Lambda(h\pm p^By,p^By):=\sum_{w=1}^{\infty}S(0,h,w)\sum_{\epsilon_1,\epsilon_2= \pm 1}\frac{(p^B,w)^{1+2i\epsilon_2r_2}}{w^{2+2i\epsilon_1r_1+2i\epsilon_2r_2}}\\ \times \zeta(1+2i\epsilon_1r_1)
\zeta(1+2i\epsilon_2r_2)
(h\pm p^By)^{i\epsilon_1r_1}y^{i\epsilon_2r_2}
\end{multline*}

and the error term is $$ET:=\frac{M^{1/2}}{N^{1/2}}\left( \frac{\sqrt{MN}}{c}\right)^{k-1}Q^{-k+1/2}Q^{5/4}(Z+N)^{1/4}(ZN)^{1/4+\epsilon}.$$

Since $Z=Q^2\frac{c^2}{M}>N$,
$$ET \ll  M^{1/2}N^{1/4}\left(
\frac{\sqrt{MN}}{c} \right)^{k-2}Q^{-k+11/4}.$$

Note that $T_{h}^{\mp}(c)$ is small when $|h| \gg Zq^{\epsilon}$ because $G^{\mp}$  is small when $z \gg Zq^{\epsilon}$. This allows us to add $\left(1+\frac{|h|}{Z} \right)^{-2}$ into the error term $ET$.
Multiplying by $S(0,h,c)$ and summing over $h$,  we have
\begin{multline*}
ET_1:=\sum_{h\neq 0}S(0,h,c)\left(1+\frac{|h|}{Z} \right)^{-2}ET \\ \ll    c^2\frac{N^{1/4}}{M^{1/2}}\left(\frac{\sqrt{MN}}{c}\right)^{k-2}Q^{-k+2+11/4}.
\end{multline*}
Finally, we sum over $c$. If $k=2$
\begin{multline*}\sum_{\substack{c \leq C\\ q|c}}c^{-2}ET_1 \ll q^{\epsilon}\frac{N^{1/4}}{M^{1/2}}\sum_{\substack{c\leq C\\ q|c}}\left[1+\left(\frac{\sqrt{MN}}{c}\right)^{11/4}\right]\\ \ll  q^{\epsilon}\left(\frac{N^{1/4}}{M^{1/2}}\frac{C}{q}+\frac{N^{13/8}M^{7/8}}{q^{11/4}}\right).
\end{multline*}
The optimal value of $C$ can be found by making equal the first summand and the error term in lemma \ref{lemma:LSI}
$$\frac{N^{1/4}}{M^{1/2}}\frac{C}{q}=\left(\frac{\sqrt{MN}}{C}\right)^{1-2\theta}.$$
Thus, $C:=\min\left(q^{\frac{1}{2-2\theta}}M^{1/2}N^{\frac{1-4\theta}{8-8\theta}}, q^{\frac{9-8\theta}{8-8\theta}}\right)$ and
$$ \sum_{M,N \ll q^{1+\epsilon}}\sum_{\substack{c \leq C\\ q|c}}c^{-2}ET_1 \ll q^{\epsilon}(q^{-\frac{1-2\theta}{8-8\theta}}+q^{-1/4}).$$
If $k \geq 4$, then
\begin{multline*}\sum_{\substack{c \leq C\\ q|c}}c^{-2}ET_1 \ll q^{\epsilon}\frac{N^{1/4}}{M^{1/2}}
\sum_{\substack{c\leq C\\ q|c}}\left[\left( \frac{\sqrt{MN}}{c}\right)^{k-2}+\left(\frac{\sqrt{MN}}{c}\right)^{11/4}\right] \\
\ll q^{\epsilon}\left(\frac{N^{1/4}}{M^{1/2}}\left(\frac{\sqrt{MN}}{q}\right)^{k-2}+\frac{N^{13/8}M^{7/8}}{q^{11/4}}\right) .
\end{multline*}
Therefore,
$$\sum_{M,N \ll q^{1+\epsilon}}\sum_{\substack{c \leq C\\ q|c}}c^{-2}ET_1\ll q^{\epsilon-1/4}.$$
Combining two estimates in one, we have that for any even $k$
$$\sum_{M,N \ll q^{1+\epsilon}}\sum_{\substack{c \leq C\\ q|c}}c^{-2}ET_1 \ll q^{\epsilon}\left(q^{-\frac{k-1-2\theta}{8-8\theta}}+q^{-1/4}\right).$$
\end{proof}

\subsection{Extension of summations}
Analogously to the off-diagonal term, at the cost of admissible error, we can reintroduce summation over $\max{(M,N)}\geq q^{1+\epsilon}$ and extend the summation over $c$ up to some large value $C_{max}=q^{\Omega}$.
\begin{prop}
For $l=0,1$, we have
\begin{equation}
\sum_{\max{(M,N)}\ll q^{1+\epsilon} }\sum_{\substack{\frac{q}{p^l}|c\\ C<c \leq C_{max} }}\frac{1}{c^2}\sum_{h \neq 0}S(0,h,c)T_{h}^{\pm}(c,B) \ll_{\epsilon, \uple{r}}  q^{\epsilon-\frac{k-1}{8-8\theta}}.
\end{equation}
\end{prop}
\begin{proof}
Consider $T_{h}^{\pm}(c,B)$ given by equation \eqref{intOOD}.
We split the sum over $w$ in expression \eqref{intOOD2} into two parts: $w<q$ and $w \geq q$.
If $w<q$ we follow section $10$ of \cite{DFI3} to show that
\begin{equation*}
\sum_{\substack{\frac{q}{p^l}|c\\ C<c \leq C_{max} }}\frac{1}{c^2}\sum_{h \neq 0}S(0,h,c)T_{h}^{\pm}(c,B) \ll_{\epsilon, \uple{r}}  q^{\epsilon}\frac{(\sqrt{MN})^k}{qC^{k-1}}.
\end{equation*}
If $w \geq q$ we estimate the absolute value of $T_{h}^{\pm}(c,B)$ using
\begin{equation*}
G^{\pm}_{r_1}(h\pm p^By,p^By)\ll \left( 1+\frac{M}{c^2}(h\pm p^By)\right)^{-2}\frac{M^{1/2}}{N^{1/2}}\left( \frac{\sqrt{MN}}{c}\right)^{k-1}
\end{equation*}
and
\begin{equation*}
S(0,w,c)\ll
(w,c ) .
\end{equation*}
This gives
\begin{equation*}
\sum_{\substack{\frac{q}{p^l}|c\\ C<c \leq C_{max} }}\frac{1}{c^2}\sum_{h \neq 0}S(0,h,c)T_{h}^{\pm}(c,B) \ll_{\epsilon, \uple{r}}  q^{\epsilon}\frac{(\sqrt{MN})^k}{qC^{k-1}}\frac{C}{qM}.
\end{equation*}
Finally, taking $C=\min\left(q^{\frac{1}{2-2\theta}}M^{1/2}N^{\frac{1-4\theta}{8-8\theta}}, q^{\frac{9-8\theta}{8-8\theta}}\right)$ and performing dyadic summation over $M,N$,
we obtain the assertion.
\end{proof}
\begin{prop}
For any $\epsilon>0$, any $A>0$ and  $l=0,1$
\begin{equation}
\sum_{\substack{\frac{q}{p^l}|c }}\frac{1}{c^2}\sum_{\max{(M,N)}\gg q^{1+\epsilon} }\sum_{h \neq 0}S(0,h,c)T_{h}^{\pm}(c,B) \ll_{\epsilon, A, \uple{r}} q^{-A}.
\end{equation}
\end{prop}
\begin{proof} The assertion follows from the rapid decay of $F_{M,N}$.
See the proof of proposition \ref{prop:A} for details.
\end{proof}

Now  it is possible to combine all functions $F_M$ into $F$ and replace $\sum_{M,N} F_{M,N}$ by
\begin{equation}
F(x,y):=\frac{1}{\sqrt{xy}}W_{t_1,r_1}(\frac{x}{\hat{q}^2})W_{t_2,r_2}(\frac{y}{\hat{q}^2})F(x)F(y),
\end{equation}
where $F(x)$ is a smooth function, compactly supported in $[1/2, \infty)$ such that $F(x)=1$ for $x \geq 1$.
\begin{lemma} One has
\begin{multline*}
M^{OOD}(B)=2\pi i^k \hat{q}^{-2t_1-2t_2}\sum_{\epsilon_1,\epsilon_2= \pm 1}\zeta(1+2i\epsilon_1r_1)
\zeta(1+2i\epsilon_2r_2) \\ \times  \left(\sum_{g,v}\frac{\mu(g)}{g^2}\frac{\mu(v)}{v^{2+2i\epsilon_1r_1+2i\epsilon_2r_2}}\sum_{\substack{q|cg\\ cg<q^{\Omega}}}\frac{1}{c}\sum_{w}\frac{(p^B,wv)^{1+2i\epsilon_2r_2}}{w^{1+2i\epsilon_1r_1+2i\epsilon_2r_2}}\sum_{\substack{h \neq 0 \\ [w,c]|h}}V(h)\right.
\\ \left. -\frac{1}{p}\sum_{g,v}\frac{\mu(g)}{g^2}\frac{\mu(v)}{v^{2+2i\epsilon_1r_1+2i\epsilon_2r_2}}\sum_{\substack{\frac{q}{p}|cg\\ cg<q^{\Omega}}}\frac{1}{c}\sum_{w}\frac{(p^B,wv)^{1+2i\epsilon_2r_2}}{w^{1+2i\epsilon_1r_1+2i\epsilon_2r_2}}\sum_{\substack{h \neq 0 \\  [w,c]|h}} V(h)\right),
\end{multline*}
where
\begin{multline*}
V(h)=-\frac{1}{(2\pi i)^2}\frac{1}{p^{B(1+i\epsilon_2r_2)}}
\int_{\Re{\beta}=0.7}\int_{\Re{z}=-0.1}\frac{\Gamma(\beta+ir_1)
\Gamma(\beta-ir_1)}{\Gamma(1+z)\Gamma(k+z)}\\
\times \frac{(4\pi)^{k+2z-2\beta}2^{-k-2z+2\beta}}{\sin{(\pi z)}}
(cg)^{-k+1-2z+2\beta}h^{k/2+z-\beta+i\epsilon_1r_1+i\epsilon_2r_2}\\
\times  \int_{x=0}^{\infty}x^{z-\beta+k/2}W_{t_1,r_1}\left(\frac{x}{\hat{q}^2}\right)F(x)\frac{dx}{x} \int_{y=0}^{\infty}y^{z+k/2+i\epsilon_2r_2}W_{t_2,r_2}\left(\frac{hy}{\hat{q}^2}\right)F(hy)\\
\times \left( \frac{\cos{(\pi\beta)}}{(1+y)^{\beta-i\epsilon_1r_1}}+\delta_{y>1}\frac{\cos{(\pi\beta)}}{(-1+y)^{\beta-i\epsilon_1r_1}}+\delta_{y<1} \frac{\cos{(\pi i r_1)}}{(1-y)^{\beta-i\epsilon_1r_1}}\right)\frac{dy}{y}dz d\beta.
\end{multline*}
\end{lemma}
\begin{proof}
Lemma \ref{choice} yields
\begin{multline*}
T_{h}^{-}(c,B)+T_{h}^{+}(c,B)=\sum_{w=1}^{\infty}S(0,h,w)\sum_{\epsilon_1,\epsilon_2= \pm 1}\frac{(p^B,w)^{1+2i\epsilon_2r_2}}{w^{2+2i\epsilon_1r_1+2i\epsilon_2r_2}} \zeta(1+2i\epsilon_1r_1)\\
\times
\zeta(1+2i\epsilon_2r_2)
\int_{0}^{\infty}\left[\delta_{h+p^By>0} G_{r_1}^{-}(h+p^By,p^By)(h+ p^By)^{i\epsilon_1r_1}y^{i\epsilon_2r_2}\right.\\
\left. +\delta_{h-p^By>0}G_{r_1}^{+}(h-p^By,p^By) (h- p^By)^{i\epsilon_1r_1}y^{i\epsilon_2r_2}\right]dy.
\end{multline*}
We plug in the expressions for $G_{r_1}^{-}$ and $G_{r_1}^{+}$ given by \eqref{eq:G-} and \eqref{eq:G+} and use the identity
$$F(x,p^By)=\frac{1}{(p^Bxy)^{1/2}}W_{t_1,r_1}\left(\frac{x}{\hat{q}^2}\right)W_{t_2,r_2}\left(\frac{p^By}{\hat{q}^2}\right)F(x)F(p^By).$$
This gives
\begin{multline*}
T_{h}^{-}(c,B)+T_{h}^{+}(c,B)=\sum_{w=1}^{\infty}S(0,h,w)\sum_{\epsilon_1,\epsilon_2= \pm 1}\frac{(p^B,w)^{1+2i\epsilon_2r_2}}{w^{2+2i\epsilon_1r_1+2i\epsilon_2r_2}} \\ \times \zeta(1+2i\epsilon_1r_1)
\zeta(1+2i\epsilon_2r_2)\\ \times
2\pi\int_{0}^{\infty}\int_{0}^{\infty}\frac{1}{(p^Bxy)^{1/2}}W_{t_1,r_1}\left(\frac{x}{\hat{q}^2}\right)W_{t_2,r_2}\left(\frac{p^By}{\hat{q}^2}\right)J_{k-1}\left(\frac{4\pi\sqrt{xp^By}}{c}\right)
\\ \times \left[\delta_{h+p^By>0} k_0\left(\frac{4\pi\sqrt{x(h+ p^By)}}{c},1/2+ir_1\right)(h+ p^By)^{i\epsilon_1r_1}y^{i\epsilon_2r_2}\right.\\
\left. +\delta_{h-p^By>0} k_1\left(\frac{4\pi\sqrt{x(h+ p^By)}}{c},1/2+ir_1\right) (h- p^By)^{i\epsilon_1r_1}y^{i\epsilon_2r_2}\right]
\\ \times F(x)F(p^By)dxdy.
\end{multline*}

The off-off-diagonal term
\begin{align*}
M^{OOD}(B)=2\pi i^k  \hat{q}^{-2t_1-2t_2}\left( \sum_{q|c}\frac{1}{c^2}\sum_{h \neq 0}S(0,h,c)(T_{h}^{-}(c,B)+T_{h}^{+}(c,B))\right.\nonumber \\ \left.
-\frac{1}{p} \sum_{\frac{q}{p}|c}\frac{1}{c^2}\sum_{h \neq 0}S(0,h,c)(T_{h}^{-}(c,B)+T_{h}^{+}(c,B)) \right)
\end{align*}
contains two Ramanujan sums $S(0,h,c)$ and $S(0,h,w)$.
Applying the formulas
 \begin{equation*}
S(0,h,c)=\sum_{\substack{gc_1=c\\ c_1|h}}\mu(g)c_1, \quad
 S(0,h,w)=\sum_{\substack{vw_1=c\\ w_1|h}}\mu(v)w_1,
\end{equation*}
we obtain
\begin{multline*}
M^{OOD}(B)=2\pi i^k \hat{q}^{-2t_1-2t_2} \sum_{\epsilon_1,\epsilon_2= \pm 1}\zeta(1+2i\epsilon_1r_1)
\zeta(1+2i\epsilon_2r_2) \\ \times  \left(\sum_{g,v}\frac{\mu(g)}{g^2}\frac{\mu(v)}{v^{2+2i\epsilon_1r_1+2i\epsilon_2r_2}}\sum_{\substack{q|cg\\ cg<q^{\Omega}}}\frac{1}{c}\sum_{w}\frac{(p^B,wv)^{1+2i\epsilon_2r_2}}{w^{1+2i\epsilon_1r_1+2i\epsilon_2r_2}}\sum_{\substack{h \neq 0 \\ [w,c]|h}}V(h)\right.
\\ \left. -\frac{1}{p}\sum_{g,v}\frac{\mu(g)}{g^2}\frac{\mu(v)}{v^{2+2i\epsilon_1r_1+2i\epsilon_2r_2}}\sum_{\substack{\frac{q}{p}|cg\\ cg<q^{\Omega}}}\frac{1}{c}\sum_{w}\frac{(p^B,wv)^{1+2i\epsilon_2r_2}}{w^{1+2i\epsilon_1r_1+2i\epsilon_2r_2}}\sum_{\substack{h \neq 0 \\  [w,c]|h}}V(h)\right),
\end{multline*}
where
\begin{multline*}
V(h)=2\pi\int_{0}^{\infty}\int_{0}^{\infty}\frac{1}{(p^Bxy)^{1/2}}W_{t_1,r_1}\left(\frac{x}{\hat{q}^2}\right)W_{t_2,r_2}\left(\frac{p^By}{\hat{q}^2}\right)
\\ \times \left[\delta_{h+p^By>0} k_0\left(\frac{4\pi\sqrt{x(h+ p^By)}}{cg},1/2+ir_1\right)(h+ p^By)^{i\epsilon_1r_1}y^{i\epsilon_2r_2}\right.\\
\left. +\delta_{h-p^By>0} k_1\left(\frac{4\pi\sqrt{x(h+ p^By)}}{cg},1/2+ir_1\right) (h- p^By)^{i\epsilon_1r_1}y^{i\epsilon_2r_2}\right]\\
\times J_{k-1}\left(\frac{4\pi\sqrt{xp^By}}{cg}\right) F(x)F(p^By)dxdy.
\end{multline*}

In the expression $V(h)$ we replace negative  $h$ by their absolute value and make the change of variables $\frac{p^By}{h} \rightarrow y $ in the integral.
As a result,
\begin{multline*}
V(h)=
2\pi
\frac{h^{1/2+i\epsilon_1r_1+i\epsilon_2r_2}}{p^{B(1+i\epsilon_2r_2)}} \int_{0}^{\infty}
\int_{0}^{\infty}
\frac{y^{i\epsilon_2r_2}}{(xy)^{1/2}}W_{t_1,r_1}\left(\frac{x}{\hat{q}^2}\right)W_{t_2,r_2}
\left(\frac{hy}{\hat{q}^2}\right)
\\ \times \Biggl[(1+ y)^{i\epsilon_1r_1}k_0\left(\frac{4\pi\sqrt{xh(1+ y)}}{cg},1/2+ir_1\right)
+\delta_{y>1}(-1+ y)^{i\epsilon_1r_1}\\
 \times k_0\left(\frac{4\pi\sqrt{xh(y-1)}}{cg},1/2+ir_1\right)+\delta_{y<1} (1- y)^{i\epsilon_1r_1}  \\ \times  k_1\left(\frac{4\pi\sqrt{xh(1-y)}}{cg},1/2+ir_1\right)\Biggr]
 J_{k-1}\left(\frac{4\pi\sqrt{xhy}}{cg}\right)  F(x)F(hy)dxdy.
\end{multline*}

Finally, we use Mellin transforms of Bessel functions \eqref{Mel1}, \eqref{Mel3} and \eqref{Mel2}, so that
\begin{multline*}
V(h)=-\frac{p^{-B(1+i\epsilon_2r_2)}}{(2\pi i)^2}
\int_{\Re{\beta}=0.7}\int_{\Re{z}=-0.1}\frac{\Gamma(\beta+ir_1)
\Gamma(\beta-ir_1)}{\Gamma(1+z)\Gamma(k+z)}\\
\times \int_{\Re{\beta}=0.7}\int_{\Re{z}=-0.1}\frac{(4\pi)^{k+2z-2\beta}2^{2\beta-k-2z}}{\sin{(\pi z)}}(cg)^{-k+1-2z+2\beta} h^{k/2+z-\beta+i\epsilon_1r_1+i\epsilon_2r_2}\\
\times \int_{x=0}^{\infty}x^{z-\beta+k/2}W_{t_1,r_1}\left(\frac{x}{\hat{q}^2}\right)F(x)\frac{dx}{x} \int_{y=0}^{\infty}y^{z+k/2+i\epsilon_2r_2}W_{t_2,r_2}\left(\frac{hy}{\hat{q}^2}\right)F(hy)\\
\times \left( \frac{\cos{(\pi\beta)}}{(1+y)^{\beta-i\epsilon_1r_1}}+\delta_{y>1}\frac{\cos{(\pi\beta)}}{(-1+y)^{\beta-i\epsilon_1r_1}}+\delta_{y<1} \frac{\cos{(\pi ir_1)}}{(1-y)^{\beta-i\epsilon_1r_1}}\right)\frac{dy}{y}dz d\beta.
\end{multline*}
Note that we shifted the contour of integration given in \eqref{Mel2} to $\Re{\beta}=0.7$, which is possible due to the rapid decay of the $x$ integral in $\beta$. The change of the order of integration in $V(h)$ is justified by absolute convergence of all integrals.
\end{proof}

\subsection{ Replacing $F(x)F(hy)$ by $1$ on the interval $[0,\infty)^2$}
This step allows us to simplify the integration and can be performed with a cost of negligible error.

\subsubsection{y-integral}

Consider
\begin{multline}
IY:=\int_{y=0}^{\infty}y^{z+k/2+i\epsilon_2r_2}W_{t_2,r_2}\left(\frac{hy}{\hat{q}^2}\right)F(hy)\\
\times \left( \frac{\cos{(\pi\beta)}}{(1+y)^{\beta-i\epsilon_1r_1}}+\delta_{y>1}\frac{\cos{(\pi\beta)}}{(-1+y)^{\beta-i\epsilon_1r_1}}+\delta_{y<1} \frac{\cos{(\pi ir_1)}}{(1-y)^{\beta-i\epsilon_1r_1}}\right)\frac{dy}{y}.
\end{multline}

\begin{lemma}\label{Fy1}
 The function $F(hy)$ can be replaced by $1$ in $IY$ with an error $$O_{\epsilon}(q^{-1/2+\epsilon}).$$
\end{lemma}
\begin{proof} $F(hy)$ is a smooth function, compactly supported in $[1/2, \infty)$ such that $F(hy)=1$ for $hy \geq 1$.
Thus, we only need to estimate the integral  for $y<1/h$ .
It is bounded by $\left(\frac{1}{h}\right)^{k/2+\Re{z}}\cos{\pi \beta}$.
We  are left to estimate
\begin{multline*} T:=\sum_{g,v,w}\frac{1}{g^2v^2w}\sum_{\substack{q|cg\\ cg<q^{\Omega}}}\frac{1}{c}
\sum_{[c,w]|h}h^{-\beta}\int_{\Re{\beta}=0.7}\int_{\Re{z}=-0.1}\frac{\Gamma(\beta+ir_1)
\Gamma(\beta-ir_1)}{\Gamma(1+z)\Gamma(k+z)}
\\ \times \frac{\cos{\pi \beta}}{\sin{(\pi z)}}(cg)^{-k+1-2z+2\beta} \int_{x=0}^{\infty}x^{z-\beta+k/2}W_{t_1,r_1}\left(\frac{x}{\hat{q}^2}\right)F(x)\frac{dx}{x}.
\end{multline*}
To make the sums over $h$ and $w$ absolutely convergent, one has to move $\beta$ contour to the right $\Re{\beta}>1$.  At the same time, partial integration shows that the $x$-integral decays rapidly in $\beta$:
\begin{multline*}
\int_{0}^{\infty}x^{z-\beta+k/2}W_{t_1,r_1}\left(\frac{x}{\hat{q}^2}\right)F(x)\frac{dx}{x}\\ =
\frac{1}{(z-\beta+k/2)(z-\beta+k/2)\ldots (z-\beta+k/2+n-1)}\\ \times \int_{0}^{\infty}\frac{\partial^n}{\partial x^{n}}\left( W_{t_1,r_1}\left(\frac{x}{\hat{q}^2}\right)F(x)\right)x^{z-\beta+k/2+n-1}dx\ll \frac{1}{|\beta|^n}q^{z-\beta+k/2}.
\end{multline*}
Assume that $\Re{\beta}>1$. We have
\begin{multline*}
T \ll q^{z-\beta+k/2}\sum_{v,w,h}\frac{1}{v^2w^{1+\beta}h^{\beta}}\sum_{\substack{q|cg\\ cg<q^{\Omega}}}\frac{1}{c^{1+\beta}g^2}(cg)^{-k+1-2z+2\beta}\\
\ll q^{z-\beta+k/2}\sum_{\substack{q|cg\\ cg<q^{\Omega}}}(cg)^{-k-1-2z+2\beta} \ll q^{z-\beta+k/2-1}q^{\Omega(-k-2z+2\beta)}.
\end{multline*}
Moving  $\beta$ contour to $\Re{\beta}=k/2+\delta$ and $z$ contour to $-\delta$, $M^{OOD}$ is dominated by $$q^{-1}q^{4\delta\Omega-2\delta}.$$
Choosing $\delta=\frac{1+2\epsilon}{4(2\Omega-1)}$, we obtain the result.
\end{proof}

\begin{lemma}
One has
\begin{multline}\label{eq:iy}
IY=\frac{1}{2 \pi i}\int_{\Re{t}=k/2-0.2}\frac{P_r(t)}{P_r(t_2)} \frac{\Gamma(t+ir_2+k/2)\Gamma(t-
ir_2+k/2)}{\Gamma(t_2+ir_2+k/2)\Gamma(t_2-ir_2+k/2)}\\
\times \left(\frac{\hat{q}^2}{h}\right)^{t}\frac{\Gamma(k/2+z-t+i\epsilon_2r_2)\Gamma(-k/2-z+t+\beta -i\epsilon_1r_1-i\epsilon_2r_2)}{\Gamma(\beta-i\epsilon_1r_1)}\\
\times \zeta_q(1+2t)
\left(\cos{(\pi\beta)}+\frac{\cos{(\pi\beta)}\sin{(\pi(k/2+z-t+i\epsilon_2r_2))}}{\sin{(\pi(\beta-i\epsilon_1r_1))}}\right.\\
\left. +\frac{\cos{(\pi ir_1)}\sin{(\pi(-k/2-z+t+\beta -i\epsilon_1r_1-i\epsilon_2r_2))}}{\sin{(\pi(\beta-i\epsilon_1r_1))}}\right)\frac{2tdt}{t^2-t_{2}^{2}}.
\end{multline}
\end{lemma}

\begin{proof}
By lemma \ref{Fy1}, the $y$-integral
is equal to
\begin{multline*}
IY=\int_{y=0}^{\infty}y^{z+k/2+i\epsilon_2r_2}W_{t_2,r_2}\left(\frac{hy}{\hat{q}^2}\right) \\
\times \left( \frac{\cos{(\pi\beta)}}{(1+y)^{\beta-i\epsilon_1r_1}}+\delta_{y>1}\frac{\cos{(\pi\beta)}}{(-1+y)^{\beta-i\epsilon_1r_1}}+\delta_{y<1} \frac{\cos{(\pi i r_1)}}{(1-y)^{\beta-i\epsilon_1r_1}}\right)\frac{dy}{y}.
\end{multline*}
We plug in the expression
\begin{multline*}
W_{t_2,r_2}\left(\frac{hy}{\hat{q}^2}\right)=\frac{1}{2\pi i}\int_{\Re{t}=k/2-0.2}\frac{P_r(t)}{P_r(t_2)}\zeta_q(1+2t)\\
\times \frac{\Gamma(t+ir_2+k/2)\Gamma(t-ir_2+k/2)}{\Gamma(t_2+ir_2+k/2)\Gamma(t_2-ir_2+k/2)}  \left(\frac{hy}{\hat{q}^2}\right)^{-t}\frac{2tdt}{t^2-t_{2}^{2}}.
\end{multline*}
 Note that  we shifted $\Re{t}$ from $3$ to $k/2-0.2$ without crossing any poles. This step is required to ensure that all poles of $\Gamma(-k/2-z+t+\beta -i\epsilon_1r_1-i\epsilon_2r_2)$ lie to the left of the $t$ contour.
Therefore,
\begin{multline*}
IY=\frac{1}{2 \pi i}\int_{\Re{t}=k/2-0.2}\frac{P_r(t)}{P_r(t_2)} \frac{\Gamma(t+ir_2+k/2)\Gamma(t-
ir_2+k/2)}{\Gamma(t_2+ir_2+k/2)\Gamma(t_2-ir_2+k/2)}\zeta_q(1+2t)
\\ \times \left(\frac{\hat{q}^2}{h}\right)^{t} \int_{y=0}^{\infty}y^{z+k/2+i\epsilon_2r_2-t}
\Biggl( \frac{\cos{(\pi\beta)}}{(1+y)^{\beta-i\epsilon_1r_1}}+\delta_{y>1}\frac{\cos{(\pi\beta)}}{(-1+y)^{\beta-i\epsilon_1r_1}}\\+
\delta_{y<1} \frac{\cos{(\pi i r_1)}}{(1-y)^{\beta-i\epsilon_1r_1}}\Biggr)\frac{dy}{y} \frac{2tdt}{t^2-t_{2}^{2}}.
\end{multline*}
Mellin transforms \eqref{M1}, \eqref{M2} and \eqref{M3} and Euler's reflection formula give
\begin{multline*}
IY=\frac{1}{2 \pi i}\int_{\Re{t}=k/2-0.2}\frac{P_r(t)}{P_r(t_2)} \frac{\Gamma(t+ir_2+k/2)\Gamma(t-
ir_2+k/2)}{\Gamma(t_2+ir_2+k/2)\Gamma(t_2-ir_2+k/2)}\zeta_q(1+2t)\\
\times \left(\frac{\hat{q}^2}{h}\right)^{t}\Biggl(\cos{(\pi\beta)}+\frac{\cos{(\pi\beta)}\sin{(\pi(k/2+z-t+i\epsilon_2r_2))}}{\sin{(\pi(\beta-i\epsilon_1r_1))}}\\+\frac{\cos{(\pi i r_1)}\sin{(\pi(-k/2-z+t+\beta -i\epsilon_1r_1-i\epsilon_2r_2))}}{\sin{(\pi(\beta-i\epsilon_1r_1))}}\Biggr)\\
\times\frac{\Gamma(k/2+z-t+i\epsilon_2r_2)\Gamma(-k/2-z+t+\beta -i\epsilon_1r_1-i\epsilon_2r_2)}{\Gamma(\beta-i\epsilon_1r_1)}
 \frac{2tdt}{t^2-t_{2}^{2}}.
\end{multline*}
\end{proof}

\begin{remark}
Consider the expression \eqref{eq:iy}.
If $\beta-i\epsilon_1r_1=0,-1,-2, \ldots,$  the poles of $1/\sin(\pi(\beta-i\epsilon_1r_1))$ are canceled by the zeroes of $1/\Gamma(\beta-i\epsilon_1r_1)$. The poles at $\beta-i\epsilon_1r_1=j \text{ with }j=1,2,3 \ldots$ are compensated by the vanishing numerator.
\end{remark}

\subsubsection{x-integral}

\begin{lemma}
The function $F(x)$ can be replaced by $1$ in the expression $V(h)$ at the cost of negligible error $O_{\epsilon, \uple{r}}(q^{\epsilon-k/2+0.5}).$
\end{lemma}
\begin{proof}
We show that the contribution of $F_1(x)=1-F(x)$ is negligible.
Note that $F_1(x)=0$ for $x\geq1$ since in that case $F(x)=1$.
 The part of $M^{OOD}$ which affects the $x$-integral can be written as follows
\begin{multline*}\sum_{v,w}\frac{1}{v^2w}\sum_{\substack{c,g \\ q|cg}}\sum_{[c,w]|h}g^{-k-1-2z+2\beta}c^{-k-2z+2\beta}h^{k/2+z-\beta-t}q^{t}\\
\times \Gamma(-k/2-z+t+\beta -i\epsilon_1r_1-i\epsilon_2r_2) \Gamma(k/2+z-t+i\epsilon_2r_2)H_1(t,z, \beta)
\\
\times \int_{0}^{1}x^{z-\beta+k/2}W_{t_1,r_1}(\frac{x}{\hat{q}^2})F_1(x)\frac{dx}{x}.
\end{multline*}
Here $H_1$ is an analytic function.
We have
\begin{equation*}
\Re{z}=-0.1\text{, }\Re{\beta}=0.7\text{, }\Re{t}=k/2-0.2.
\end{equation*}
Without crossing any poles, we shift $\beta$-contour to $$\Re{\beta}=0.3.$$
 In order to make the sums over $h$ and $w$ absolutely convergent, we move the $t$ contour to $$\Re{t}=k/2+0.7,$$
crossing a pole at $t=k/2+z+i\epsilon_2r_2$.
Since $\Re{z}-\Re{\beta}+k/2>0$, the $x$-integral can be done by parts $n$ times (for sufficiently large $n$) to make $\beta$-integral convergent. This gives
\begin{equation*}
\int_{0}^{1}x^{z-\beta+k/2}W_{t_1,r_1}(\frac{x}{\hat{q}^2})F_1(x)\frac{dx}{x}\ll  \frac{1}{|\beta|^n}.
\end{equation*}
Finally,  all sums and integrals are absolutely convergent and $q^{-k-1+t-2z+2\beta}$ can be factored out due to divisibility conditions. In total, this gives an error $O_{\epsilon, \uple{r}}(q^{\epsilon-k/2+0.5}).$

For the pole at $t=k/2+z+i\epsilon_2r_2$ another contour shift is required to make all sums absolutely convergent. We move the $z$-contour to $$\Re{z}=0.5+2\epsilon$$ and $\beta$ to $$\Re{\beta}=1+\epsilon.$$ Note that the pole of $1/\sin{(\pi z)}$ at $z=0$ is canceled by zero of $P_r(t)=P_r(k/2+z+i\epsilon_2r_2)$. The $x$ integral is bounded by
$\frac{1}{|\beta|^n}$. The power of $q$, corresponding to divisibility conditions on $g,c,h$, is $q^{-k-1+t-2z+2\beta}.$ This gives the error term $O_{\epsilon,\uple{r}}(q^{\epsilon-k/2+0.5}).$
\end{proof}

\begin{prop} One has
\begin{multline*}
M^{OOD}(B)=2\pi i^k \hat{q}^{-2t_1-2t_2}\sum_{\epsilon_1,\epsilon_2= \pm 1}\zeta(1+2i\epsilon_1r_1)
\zeta(1+2i\epsilon_2r_2) \\ \times  \left(\sum_{g,v}\frac{\mu(g)}{g^2}\frac{\mu(v)}{v^{2+2i\epsilon_1r_1+2i\epsilon_2r_2}}\sum_{\substack{q|cg\\ cg<q^{\Omega}}}\frac{1}{c}\sum_{w}\frac{(p^B,wv)^{1+2i\epsilon_2r_2}}{w^{1+2i\epsilon_1r_1+2i\epsilon_2r_2}}\sum_{\substack{h \neq 0 \\ [w,c]|h}}V(h)\right.
\\ \left. -\frac{1}{p}\sum_{g,v}\frac{\mu(g)}{g^2}\frac{\mu(v)}{v^{2+2i\epsilon_1r_1+2i\epsilon_2r_2}}\sum_{\substack{\frac{q}{p}|cg\\ cg<q^{\Omega}}}\frac{1}{c}\sum_{w}\frac{(p^B,wv)^{1+2i\epsilon_2r_2}}{w^{1+2i\epsilon_1r_1+2i\epsilon_2r_2}}\sum_{\substack{h \neq 0 \\  [w,c]|h}}V(h)\right),
\end{multline*}
where
\begin{multline*}
V(h)=-\frac{i^k}{(2\pi i)^3}\int_{\Re{t}=k/2+0.7}\int_{\Re{s}=k/2-0.4}\int_{\Re{z}=-0.1}\frac{\hat{q}^{2s+2t}}{p^{B(1+i\epsilon_2r_2)}}(cg)^{1-2s}\\ \times (2\pi)^{2s}
\frac{P_r(s)P_r(t)}{P_r(t_1)P_r(t_2)}\zeta_q(1+2t)\zeta_q(1+2s)
\frac{h^{s-t+i\epsilon_1r_1+i\epsilon_2r_2}}{\sin{(\pi z)}\Gamma(1+z)\Gamma(k+z)}\\
\times \frac{\Gamma(k/2+s\pm ir_1)\Gamma(k/2+t\pm ir_2)}{\Gamma(k/2+t_1\pm ir_1)\Gamma(k/2+t_2\pm ir_2)}\\
\times \Gamma(t-s-i\epsilon_1r_1-i\epsilon_2r_2)\Gamma(k/2+z-s+ i\epsilon_1r_1)\Gamma(k/2+z-t+ i\epsilon_2r_2)
\\ \times \left(\cos{(\pi(z-s))}+\frac{\cos{(\pi(z-s))}\sin{(\pi(z-t+i\epsilon_2r_2))}}{\sin{(\pi(z-s-i\epsilon_1r_1))}}\right.
\\ \left.+\frac{\cos{(\pi i r_1)}\sin{(\pi(t-s-i\epsilon_1r_1-i\epsilon_2r_2))}}{\sin{(\pi(z-s-i\epsilon_1r_1))}}\right)
dz\frac{2sds}{s^2-t_{1}^{2}}\frac{2tdt}{t^2-t_{2}^{2}}.
\end{multline*}
\end{prop}
\begin{remark} \label{poles}We do not compute the contribution of poles at $t=k/2+z+i\epsilon_2r_2$ since it will be cancelled by another contour shift in \ref{poles2}.
\end{remark}
\begin{proof}
By inverse Mellin transform we have
\begin{align*}
\int_{0}^{\infty}x^{z-1-\beta+k/2}W_{t_1,r_1}(\frac{x}{\hat{q}^2})dx
= 2\frac{P_r(z-\beta+k/2)}{P_r(t_1)}\zeta_q(1+k+2z-2\beta)\\
\times \hat{q}^{k+2z-2\beta} \frac{\Gamma(k+z-\beta+ir_1)\Gamma(k+z-\beta-ir_1)}{\Gamma(k/2+t_1+ir_1)\Gamma(k/2+t_1-ir_1)} \frac{k/2+z-\beta}{(k/2+z-\beta)^2-t_{1}^{2}}
\end{align*} for $\Re{(z-\beta+k/2)}>-1$.
Setting $s:=k/2+z-\beta$ yields the assertion.
\end{proof}

\subsection{Shifting the z-contour}
The $z$-integral is given by
\begin{multline}
IZ:=\frac{1}{2\pi i}\int_{\Re{z}=-0.1}
\frac{\Gamma(k/2+z-s+i\epsilon_1r_1)\Gamma(k/2+z-t+ i\epsilon_2r_2)}{\sin{(\pi z)}\Gamma(1+z)\Gamma(k+z)}
\\ \times \left(\cos{(\pi(z-s))}+\frac{\cos{(\pi(z-s))}\sin{(\pi(z-t+i\epsilon_2r_2))}}{\sin{(\pi(z-s-i\epsilon_1r_1))}}\right.
\\\left.+\frac{\cos{(\pi i r_1)}\sin{(\pi(t-s-i\epsilon_1r_1-i\epsilon_2r_2))}}{\sin{(\pi(z-s-i\epsilon_1r_1))}}\right)
dz.
\end{multline}
Stirling's formula implies that the integrand decays as $|z|^{-1-s-t}$. We shift $\Re{z}$ to $D>0$ and let $D \rightarrow +\infty$. This leads to three types of possible poles described in the table below.

\begin{center}
    \begin{tabular}  {| l | l |  }
    \hline
    Possible poles at &  Coming from function   \\ \hline
    $z=t-k/2-i\epsilon_2r_2$ & $\Gamma(k/2+z-t+i\epsilon_2r_2)$
     \\ \hline
    $z=n+s+i\epsilon_1r_1$& $1/\sin{(\pi(z-s-i\epsilon_1r_1))}$  \\ \hline
    $z=n, n \geq 0$ & $1/\sin{(\pi z)}$   \\ \hline
    \end{tabular}
\end{center}

\subsubsection{ Poles at $z=t-k/2-i\epsilon_2r_2$}\label{poles2}
The residues at these poles cancel those mentioned in remark \ref{poles} (while performing the shift of $t$ to the right). Consider
$$\int_{t}\int_{z}\Gamma(k/2+z-t+i \epsilon_2r_2)f(z,t)dzdt.$$
Shifting $t$ integral to the right, we have the residue
$$-Res_{z=-k/2+t-i \epsilon_2r_2}\Gamma(k/2+z-t+i \epsilon_2r_2)f(z,t).$$  Moving $z$ to the right, we obtain
$$-Res_{t=k/2+z+i \epsilon_2r_2}\Gamma(k/2+z-t+i \epsilon_2r_2)f(z,t).$$
Since $z$ and $t$ have different signs in $\Gamma(k/2+z-t+i \epsilon_2r_2)$, these residues cancel each other.

\subsubsection{ Poles at $z=n+s+i\epsilon_1r_1$}
\begin{prop}
The expression under the integral in $IZ$ is holomorphic at $z=n+s+i\epsilon_1r_1$.
\end{prop}
\begin{proof}
To show this, we write
\begin{multline*}\sin{(\pi(z-t+i\epsilon_2r_2))}=-\sin{(\pi(t-s-i\epsilon_1r_1-i\epsilon_2r_2))}\\
\times \cos{(\pi(z-s-i\epsilon_1r_1))}+\cos{(\pi(t-s-i\epsilon_1r_1-i\epsilon_2r_2))}\sin{(\pi(z-s-i\epsilon_1r_1))}
\end{multline*}
 and  plug it in $IZ$. After simplifications,
\begin{multline*}
IZ=\frac{1}{2\pi i}\int_{\Re{z}=-0.1}
\frac{\Gamma(k/2+z-s+i\epsilon_1r_1)\Gamma(k/2+z-t+ i\epsilon_2r_2)}{\sin{(\pi z)}\Gamma(1+z)\Gamma(k+z)}\\
\times
\Biggl[ \cos(\pi(z-s)) +\cos(\pi(z-s))\cos(\pi(t-s-i\epsilon_1r_1-i\epsilon_2r_2))\\
+\sin(\pi(z-s))\sin(\pi(z-s+i\epsilon_1r_1))\Biggr].
\end{multline*}
This is holomorphic at $z=n+s+i\epsilon_1r_1$.
\end{proof}
\subsubsection{ Poles at $z=n, n\geq 0$}\label{poleszn}
\begin{prop} The poles at $z=n$ are simple and its contribution to $IZ$ is given by
\begin{multline*}
-\frac{1}{\pi}\Gamma(s+t-i\epsilon_1r_1-i\epsilon_2r_2)\frac{\Gamma(k/2-s+i\epsilon_1r_1)\Gamma(k/2-t+i\epsilon_2r_2)}{\Gamma(k/2+s-i\epsilon_1r_1)\Gamma(k/2+t-i\epsilon_2r_2)}\\
\times \left[\cos{(\pi s)}+\cos{(\pi(t-i\epsilon_1r_1-i\epsilon_2r_2))}\right].
\end{multline*}
\end{prop}
\begin{proof}
Consider
\begin{multline*}
P_1:=-\frac{1}{\pi}\sum_{n=0}^{\infty}\frac{1}{\cos{(\pi n)}}
\frac{\Gamma(k/2+n-s+i\epsilon_1r_1)\Gamma(k/2+n-t+ i\epsilon_2r_2)}{\Gamma(1+n)\Gamma(k+n)}
\\ \times \Biggl(\cos{(\pi(n-s))}+\frac{\cos{(\pi(n-s))}\sin{(\pi(n-t+i\epsilon_2r_2))}}{\sin{(\pi(n-s-i\epsilon_1r_1))}}\\+\frac{\cos{(\pi i  r_1)}\sin{(\pi(t-s-i\epsilon_1r_1-i\epsilon_2r_2))}}{\sin{(\pi(n-s-i\epsilon_1r_1))}}\Biggr).
\end{multline*}
Since $n \in \Z$, we have
\begin{multline*}
P_1:=-\frac{1}{\pi}\sum_{n=0}^{\infty}
\frac{\Gamma(k/2+n-s+i\epsilon_1r_1)\Gamma(k/2+n-t+ i\epsilon_2r_2)}{\Gamma(1+n)\Gamma(k+n)}
\\ \times \Biggl(\cos{(\pi s)}+\frac{\cos{(\pi s)}\sin{(\pi(t-i\epsilon_2r_2))}}{\sin{(\pi(s+i\epsilon_1r_1))}}\\-\frac{\cos{(\pi i r_1)}\sin{(\pi(t-s-i\epsilon_1r_1-i\epsilon_2r_2))}}{\sin{(\pi(s+i\epsilon_1r_1))}}\Biggr).
\end{multline*}

By Gauss hypergeometric identity,
\begin{multline*}
\sum_{n=0}^{\infty}\frac{\Gamma(k/2+n-s+i\epsilon_1r_1)\Gamma(k/2+n-t+ i\epsilon_2r_2)}{\Gamma(1+n)\Gamma(k+n)}\\=\Gamma(s+t-i\epsilon_1r_1-i\epsilon_2r_2)\frac{\Gamma(k/2-s+i\epsilon_1r_1)\Gamma(k/2-t+i\epsilon_2r_2)}{\Gamma(k/2+s-i\epsilon_1r_1)\Gamma(k/2+t-i\epsilon_2r_2)}.
\end{multline*}
Simplifying the trigonometric part, we obtain
\begin{multline*}
\frac{\cos{(\pi s)}\sin{(\pi(t-i\epsilon_2r_2))}}{\sin{(\pi(s+i\epsilon_1r_1))}}
-\frac{\cos{(\pi i r_1)}\sin{(\pi(t-s-i\epsilon_1r_1-i\epsilon_2r_2))}}{\sin{(\pi(s+i\epsilon_1r_1))}}\\=
\cos{(\pi(t-i\epsilon_1r_1-i\epsilon_2r_2))}.
\end{multline*}
This implies
\begin{multline*}
P_1= -\frac{1}{\pi}\Gamma(s+t-i\epsilon_1r_1-i\epsilon_2r_2)\frac{\Gamma(k/2-s+i\epsilon_1r_1)\Gamma(k/2-t+i\epsilon_2r_2)}{\Gamma(k/2+s-i\epsilon_1r_1)\Gamma(k/2+t-i\epsilon_2r_2)}\\
\times \left[\cos{(\pi s)}+\cos{(\pi(t-i\epsilon_1r_1-i\epsilon_2r_2))}\right].
\end{multline*}
\end{proof}

\begin{prop}\label{propmood}
The off-off-diagonal term can be written as follows
\begin{multline*}
M^{OOD}(B)=\frac{2}{(2\pi i)^2}\hat{q}^{-2t_1-2t_2}\sum_{\epsilon_1,\epsilon_2=\pm 1}\zeta(1+2i\epsilon_1r_1)\zeta(1+2i\epsilon_2r_2)\\
\times \int_{\Re{t}=k/2+0.7}\int_{\Re{s}=k/2-0.4}\frac{P_r(s)P_r(t)}{P_r(t_1)P_r(t_2)}\zeta_{q}(1+2s)\zeta_{q}(1+2t)
\frac{q^{s+t}}{(2\pi)^{2t}}
\\ \times \Gamma(t-s-i\epsilon_1r_1-i\epsilon_2r_2)\Gamma(t+s-i\epsilon_1r_1-i\epsilon_2r_2)
\\
\times \frac{\Gamma(k/2+s +i\epsilon_1r_1)\Gamma(k/2+t +i\epsilon_2r_2)\Gamma(k/2-s +i\epsilon_1r_1)}{\Gamma(k/2+t_1+ir_1)\Gamma(k/2+t_1-ir_1)\Gamma(k/2+t_2+ir_2)}\\
\times \frac{\Gamma(k/2-t +i\epsilon_2r_2)}{\Gamma(k/2+t_2-ir_2)} \left(\sum_{q|cg}\sum_{g}\frac{\mu(g)}{g^{2s+1}}TD(c)-1/p\sum_{q|cpg}\sum_{g}\frac{\mu(g)}{g^{2s+1}}TD(c)\right)\\
\times
\left[\cos{(\pi s)}+\cos{(\pi (t-i\epsilon_1r_1-i\epsilon_2r_2))}\right]\frac{2sds}{s^2-t_{1}^{2}}\frac{2tdt}{t^2-t_{2}^{2}},
\end{multline*}
where
\begin{multline} \label{eq:TD}
TD(c)=\frac{1}{c^{2s}}\sum_{v}\frac{\mu(v)}{v^{2+2i\epsilon_1r_1+2i\epsilon_2r_2}}
\sum_{w}\frac{(p^B,wv)^{1+2i\epsilon_2r_2}}{p^{B(1+i\epsilon_2r_2)}w^{1+2i\epsilon_1r_1+2i\epsilon_2r_2}}\\
\times \sum_{c,w|h}\frac{1}{h^{t-s-i\epsilon_1r_1-i\epsilon_2r_2}}.
\end{multline}
\end{prop}

\subsection{Explicit formula}
We start with transforming the off-off-diagonal term.

\begin{prop}
One has
\begin{multline}\label{eqMOOD}
M^{OOD}=\sum_{\epsilon_1,\epsilon_2=\pm 1}\zeta(1+2i\epsilon_1r_1)\zeta(1+2i\epsilon_2r_2)
\frac{1}{(2\pi i)^2}\\ \times \int_{\Re{t}=k/2+0.6}\int_{\Re{s}=k/2-0.4}E(s,t)\Phi(s,t)2sds2tdt,
\end{multline}
where
\begin{multline}\label{eqE}
E(s,t):=\hat{q}^{-2t_1-2t_2}\frac{P_r(s)P_r(t)}{P_r(t_1)P_r(t_2)}\frac{1}{s^2-t_{1}^{2}}\frac{1}{t^2-t_{2}^{2}} \frac{\Gamma(k/2+s +i\epsilon_1r_1)}{\Gamma(k/2+t_1+ir_1)}\\
\times \frac{\Gamma(k/2+t +i\epsilon_2r_2)\Gamma(k/2-s +i\epsilon_1r_1)\Gamma(k/2-t +i\epsilon_2r_2)}{\Gamma(k/2+t_1-ir_1)\Gamma(k/2+t_2+ir_2)\Gamma(k/2+t_2-ir_2)},
\end{multline}
\begin{multline}\label{eq:Fst}
\Phi(s,t):=2\zeta_{q}(1+2t)
\frac{q^{s+t}}{(2\pi)^{2t}}\left[\cos{(\pi s)}+\cos{(\pi (t-i\epsilon_1r_1-i\epsilon_2r_2))}\right]\\
\times
\Gamma(t-s-i\epsilon_1r_1-i\epsilon_2r_2)\Gamma(t+s-i\epsilon_1r_1-i\epsilon_2r_2)\sum_{A,B=0}^{2}C(A,B)\sum_{q|cp^A}TD(c)
\end{multline}
and coefficients $C(A,B)$ are given in table \ref{tab:coeff}.
\end{prop}
\begin{proof}
Consider the term $M^{OOD}(B)$.
M\"{o}bius function does not vanish only if $(q,g)=1$ or $(q,g)=p$.
Then we can write
\begin{multline*}
M^{OOD}(B)=\frac{2}{(2\pi i)^2}\hat{q}^{-2t_1-2t_2}\sum_{\epsilon_1,\epsilon_2=\pm 1}\zeta(1+2i\epsilon_1r_1)\zeta(1+2i\epsilon_2r_2)\\
\times \int_{\Re{t}=k/2+0.7}\int_{\Re{s}=k/2-0.4}\frac{P_r(s)P_r(t)}{P_r(t_1)P_r(t_2)}\zeta_{q}(1+2s)\zeta_{q}(1+2t)
\frac{q^{s+t}}{(2\pi)^{2t}}\\
\times \frac{\Gamma(k/2+s +i\epsilon_1r_1)\Gamma(k/2+t +i\epsilon_2r_2)\Gamma(k/2-s +i\epsilon_1r_1)\Gamma(k/2-t +i\epsilon_2r_2)}{\Gamma(k/2+t_1+ir_1)\Gamma(k/2+t_1-ir_1)\Gamma(k/2+t_2+ir_2)\Gamma(k/2+t_2-ir_2)}\\
\times
\sum_{(q,g)=1}\frac{\mu(g)}{g^{1+2s}}\left[\sum_{q|c}TD(c)-(\frac{1}{p^{1+2s}}+\frac{1}{p})\sum_{q|cp}TD(c)+\frac{1}{p^{2+2s}}\sum_{q|cp^2}TD(c)\right]
\\ \times
\Gamma(t-s-i\epsilon_1r_1-i\epsilon_2r_2)\Gamma(t+s-i\epsilon_1r_1-i\epsilon_2r_2)\\
\times
\left[\cos{(\pi s)}+\cos{(\pi (t-i\epsilon_1r_1-i\epsilon_2r_2))}\right]\frac{2sds}{s^2-t_{1}^{2}}\frac{2tdt}{t^2-t_{2}^{2}}.
\end{multline*}
Note that $$\zeta_q(1+2s) \sum_{(q,g)=1}\frac{\mu(g)}{g^{1+2s}}=1.$$
In order to simplify notations let us denote
\begin{multline*}
E(s,t):=\hat{q}^{-2t_1-2t_2}\frac{P_r(s)P_r(t)}{P_r(t_1)P_r(t_2)}\frac{1}{s^2-t_{1}^{2}}\frac{1}{t^2-t_{2}^{2}}\frac{\Gamma(k/2+s +i\epsilon_1r_1)}{\Gamma(k/2+t_1+ir_1)}\\
\times \frac{\Gamma(k/2+t +i\epsilon_2r_2)\Gamma(k/2-s +i\epsilon_1r_1)\Gamma(k/2-t +i\epsilon_2r_2)}{\Gamma(k/2+t_1-ir_1)\Gamma(k/2+t_2+ir_2)\Gamma(k/2+t_2-ir_2)}.
\end{multline*}
This is an even function since $G$ is even.
By equation \eqref{EOOD}
\begin{equation*}
M^{OOD}=M^{OOD}(0)-\tau_{1/2+ir_2}(p)M^{OOD}(1)+M^{OOD}(2).
\end{equation*}
Next, we introduce parameter $A$ corresponding to the
condition $q|cp^A$. So that
\begin{equation*}
M^{OOD}=\sum_{A,B=0}^{2}C(A,B)M^{OOD}(A,B),
\end{equation*}
\begin{multline*}
M^{OOD}(A,B)=\frac{2}{(2\pi i)^2}\sum_{\epsilon_1,\epsilon_2=\pm 1}\zeta(1+2i\epsilon_1r_1)\zeta(1+2i\epsilon_2r_2)\\ \times
\int_{\Re{t}=k/2+0.7}\int_{\Re{s}=k/2-0.4}E(s,t)\zeta_{q}(1+2t)
\frac{q^{s+t}}{(2\pi)^{2t}}\\
\times \left[\cos{(\pi s)}+\cos{(\pi (t-i\epsilon_1r_1-i\epsilon_2r_2))}\right]\\ \times
\Gamma(t-s-i\epsilon_1r_1-i\epsilon_2r_2)\Gamma(t+s-i\epsilon_1r_1-i\epsilon_2r_2)
\sum_{q|cp^A}TD(c)2sds2tdt,
\end{multline*}
where coefficients $C(A,B)$ are given in the table \ref{tab:coeff}.

\end{proof}

\begin{table}[position specifier]
\centering
\begin{tabular}{{ | l | l | l | l |}}\hline
& $A=0$ & $A=1$ & $A=2$\\ \hline
   $ B=0 $&$ 1 $& $-(1+p^{2s})p^{-2s-1}$& $ p^{-2-2s}$\\ \hline
    $ B=1 $& $-\tau_{1/2+ir_2}(p)$& $\tau_{1/2+ir_2}(p)(1+p^{2s})p^{-2s-1}$ & $-\tau_{1/2+ir_2}(p) p^{-2-2s}$\\ \hline
    $ B=2 $ & $1$& $-(1+p^{2s})p^{-2s-1}$ &$ p^{-2-2s}$\\
    \hline
\end{tabular}
\caption{Values of coefficients C(A,B)}
\label{tab:coeff}
\end{table}

The next lemma allows removing the divisibility condition $c,w|h$  in the expression $\sum_{q|cp^A}TD(c)$.
\begin{lemma}\label{divis}
One has
\begin{multline}\sum_{\substack{c,w\\p^{\nu-A}|c}} f(c,w)\sum_{c,w| h} g(h)
=\sum_{p\nmid u}\mu(u)\\
\times\sum_{\substack{ \beta \geq 0 \\ \delta= \max(\nu-A,\beta)}} \sum_{\substack{c\\p\nmid d\\ p \nmid w}}f(p^{\nu-A} duc,p^{\beta}duw)  \sum_{h}g(p^{\delta}u^2dcwh).
\end{multline}
\end{lemma}
\begin{remark}
Recall that $q=p^{\nu}$, $\nu \geq 3$ and so $\nu-A\geq 1$.
\end{remark}
\begin{proof}
Consider $$S:=\sum_{\substack{c,w\\p^{\nu-A}|c}} f(c,w)\sum_{c,w|h}g(h).$$
Let us make the following change of variables
$$c=p^{\nu-A}c_1=p^{\nu-A}dc_2,$$
$$w=p^{\beta}w_1=p^{\beta}dw_2,$$
$$d=(c_1,w_1)\text{ so that }(c_2,w_2)=1 \text{ and }p \nmid dw_2 ,$$
$$h=p^{\delta}dc_2w_2h_1\text{ where }\delta= \max(\nu-A,\beta).$$
Then
$$S= \sum_{\substack{ \beta \geq 0\\ \delta =\max(\nu-A,\beta)}} \sum_{\substack{p\nmid d} }\sum_{\substack{c_2\\p \nmid w_2 \\(c_2,w_2)=1 }}f(p^{\nu-A}dc_2,p^{\beta}dw_2)\sum_{h_1}g(p^{\delta}dc_2w_2h_1).$$
Finally, we remove the requirement $(c_2,w_2)=1 $ by M\"{o}bius inversion
$$S=\sum_{p\nmid u}\mu(u)\sum_{\substack{ \beta \geq 0 \\ \delta=\max(\nu-A,\beta)}} \sum_{\substack{c\\p\nmid d\\p\nmid w}}f(p^{\nu-A} duc,p^{\beta}duw) \sum_{ h}g(p^{\delta}u^2dcwh).$$

\end{proof}

\begin{prop}
We have
\begin{multline*}
\Phi(s,t)=q^{t-s}(2\pi)^{-2i\epsilon_1r_1-2i \epsilon_2r_2}\frac{\zeta_q(1+t+s+i\epsilon_1r_1+i\epsilon_2r_2)}{\zeta_q(2+2i\epsilon_1r_1+2i\epsilon_2r_2)}\\
\times  \zeta(1-t+s+i\epsilon_1r_1+i\epsilon_2r_2)\zeta(1-t-s+i\epsilon_1r_1+i\epsilon_2r_2)\\
\times \zeta_q(1+t-s+i\epsilon_1r_1+i\epsilon_2r_2) \sum_{\alpha \geq 0} \frac{\mu(p^{\alpha})}{p^{\alpha(2+2i\epsilon_1r_1+2i\epsilon_2r_2)}}\sum_{A,B=0}^{2}C(A,B)(p^A)^{2s}\\
\times
\sum_{\substack{\beta \geq 0\\ \delta=\max(\nu-A,\beta)}}\frac{(p^B,p^{\alpha+\beta})^{1+2i\epsilon_2r_2}}{p^{B(1+i\epsilon_2r_2)}p^{\beta(1+2i\epsilon_1r_1+2i\epsilon_2r_2)}p^{\delta(t-s-i\epsilon_1r_1-i\epsilon_2r_2)}}.
\end{multline*}
\end{prop}
\begin{proof}
The expression $TD(c)$ is given by \eqref{eq:TD}.
Consider
\begin{multline*}
\sum_{p^{\nu-A}|c}TD(c)=\sum_{p^{\nu-A}|c}\frac{1}{c^{2s}}\sum_{v}\frac{\mu(v)}{v^{2+2i\epsilon_1r_1+2i\epsilon_2r_2}}
\sum_{w}\frac{(p^B,wv)^{1+2i\epsilon_2r_2}}{p^{B(1+i\epsilon_2r_2)}w^{1+2i\epsilon_1r_1+2i\epsilon_2r_2}}\\
\times \sum_{c,w|h}\frac{1}{h^{t-s-i\epsilon_1r_1-i\epsilon_2r_2}}.
\end{multline*}
According to the lemma \ref{divis}
$$c \rightarrow p^{\nu-A} duc,$$
$$w \rightarrow p^{\beta} duw,$$
$$h \rightarrow p^{\delta}u^2dcwh.$$
In addition,  the sum over $v$ can be decomposed as
$$\sum_{v}\frac{\mu(v)}{v^{2+2i\epsilon_1r_1+2i\epsilon_2r_2}}=\sum_{\alpha \geq 0} \frac{\mu(p^{\alpha})}{p^{\alpha(2+2i\epsilon_1r_1+2i\epsilon_2r_2)}}\sum_{(v,p)=1}\frac{\mu(v)}{v^{2+2i\epsilon_1r_1+2i\epsilon_2r_2}}.$$
Then
\begin{multline*}
\sum_{q|cp^A}TD(c)=\left(\frac{p^A}{q}\right)^{2s}\sum_{(u,p)=1}\frac{\mu(u)}{u^{2t+1}} \sum_{(v,p)=1}\frac{\mu(v)}{v^{2+2i\epsilon_1r_1+2i\epsilon_2r_2}}\\
\times \sum_{c}\frac{1}{c^{t+s-i\epsilon_1r_1-i\epsilon_2r_2}}\sum_{h}\frac{1}{h^{t-s-i\epsilon_1r_1-i\epsilon_2r_2}}\sum_{(d,p)=1}\frac{1}{d^{1+t+s+i\epsilon_1r_+i\epsilon_2r_2}}\\
\times
\sum_{(w,p)=1}\frac{1}{w^{1+t-s+i\epsilon_1r_+i\epsilon_2r_2}}
\sum_{\alpha \geq 0} \frac{\mu(p^{\alpha})}{p^{\alpha(2+2i\epsilon_1r_1+2i\epsilon_2r_2)}}\\
\times \sum_{\substack{\beta \geq 0\\ \delta=\max(\nu-A,\beta)}}\frac{(p^B,p^{\alpha+\beta})^{1+2i\epsilon_2r_2}}{p^{B(1+i\epsilon_2r_2)}p^{\beta(1+2i\epsilon_1r_1+2i\epsilon_2r_2)}p^{\delta(t-s-i\epsilon_1r_1-i\epsilon_2r_2)}}.
\end{multline*}
The asymmetric functional equation implies
\begin{multline*}
\frac{\Gamma(t-s-i\epsilon_1r_1-i\epsilon_2r_2)\Gamma(t+s-i\epsilon_1r_1-i\epsilon_2r_2)\prod \zeta(t\pm s-i\epsilon_1r_1-i\epsilon_2r_2)}{(2\pi)^{2t-2i\epsilon_1r_1-2i\epsilon_2r_2}}\\ =
\frac{\zeta(1-t-s+i\epsilon_1r_1+i\epsilon_2r_2)\zeta(1-t+s+i\epsilon_1r_1+i\epsilon_2r_2)}{2\left[\cos{(\pi s)}+\cos{(\pi (t-i\epsilon_1r_1-i\epsilon_2r_2))}\right] }.
\end{multline*}
Thus,
\begin{multline*}
\Phi(s,t)=q^{t-s}(2\pi)^{-2i\epsilon_1r_1-2i \epsilon_2r_2}\frac{\zeta_q(1+t+s+i\epsilon_1r_1+i\epsilon_2r_2)}{\zeta_q(2+2i\epsilon_1r_1+2i\epsilon_2r_2)}\\
\times \zeta_q(1+t-s+i\epsilon_1r_1+i\epsilon_2r_2)\zeta(1-t+s+i\epsilon_1r_1+i\epsilon_2r_2)\zeta(1-t-s+i\epsilon_1r_1+i\epsilon_2r_2)\\
\times \sum_{\alpha \geq 0} \frac{\mu(p^{\alpha})}{p^{\alpha(2+2i\epsilon_1r_1+2i\epsilon_2r_2)}}
 \sum_{A,B=0}^{2}C(A,B)(p^A)^{2s} \\
\times
\sum_{\substack{\beta \geq 0\\ \delta=\max(\nu-A,\beta)}}\frac{(p^B,p^{\alpha+\beta})^{1+2i\epsilon_2r_2}}{p^{B(1+i\epsilon_2r_2)}p^{\beta(1+2i\epsilon_1r_1+2i\epsilon_2r_2)}p^{\delta(t-s-i\epsilon_1r_1-i\epsilon_2r_2)}}.
\end{multline*}
\end{proof}

The sums over $\alpha$ and $\beta$ in $\Phi(s,t)$ can be evaluated by considering different cases, as we now show.

\subsubsection{ Case 1: $\beta> \nu-A$}

\begin{prop}
The given case contributes to the off-off-diagonal term as $O_{\epsilon,\uple{r}}(q^{-1+\epsilon})$.
\end{prop}
\begin{proof}
We have $\delta=\beta$ and
\begin{multline*}
\Phi(s,t)=q^{t-s}(2\pi)^{-2i\epsilon_1r_1-2i \epsilon_2r_2}\frac{\zeta_q(1+t+s+i\epsilon_1r_1+i\epsilon_2r_2)}{\zeta_q(2+2i\epsilon_1r_1+2i\epsilon_2r_2)}\\
\times  \zeta(1-t+s+i\epsilon_1r_1+i\epsilon_2r_2)\zeta(1-t-s+i\epsilon_1r_1+i\epsilon_2r_2)\\
\times \zeta_q(1+t-s+i\epsilon_1r_1+i\epsilon_2r_2) \sum_{\alpha \geq 0} \frac{\mu(p^{\alpha})}{p^{\alpha(2+2i\epsilon_1r_1+2i\epsilon_2r_2)}} \sum_{A,B=0}^{2}C(A,B)(p^A)^{2s}\\
\times
\sum_{\beta \geq \nu -A+1}\frac{(p^B,p^{\alpha+\beta})^{1+2i\epsilon_2r_2}}{p^{B(1+i\epsilon_2r_2)}p^{\beta(1+t-s+i\epsilon_1r_1+i\epsilon_2r_2)}}.
\end{multline*}
The sum over $\beta$ is given by
\begin{multline*}
q^{t-s}\sum_{\beta \geq \nu-A+1}\frac{1}{(p^{1+t-s+i\epsilon_1r_1+i\epsilon_2r_2})^{\beta}}
=\frac{1}{q}(p^{A-1})^{1+t-s+i\epsilon_1r_1+i\epsilon_2r_2}\\
\times
\sum_{\beta \geq 0}\frac{1}{(p^{1+t-s+i\epsilon_1r_1+i\epsilon_2r_2})^{\beta}}.
\end{multline*}
This implies that the contribution of this case to $M^{OOD}$ is bounded by $O_{\epsilon,\uple{r}}(q^{-1+\epsilon}).$
\end{proof}

\subsubsection{ Case 2: $\beta \leq \nu-A$ }
The condition   $\beta \leq \nu-A$ means that $\delta= \nu-A$ and
\begin{multline*}
\Phi(s,t)=\hat{q}^{2i\epsilon_1r_1+2i\epsilon_2r_2}\frac{\zeta_q(1+t+s+i\epsilon_1r_1+i\epsilon_2r_2)}{\zeta_q(2+2i\epsilon_1r_1+2i\epsilon_2r_2)}\\
\times \zeta(1-t+s+i\epsilon_1r_1+i\epsilon_2r_2)\zeta(1-t-s+i\epsilon_1r_1+i\epsilon_2r_2)\\
\times \zeta_q(1+t-s+i\epsilon_1r_1+i\epsilon_2r_2)\sum_{\alpha \geq 0} \frac{\mu(p^{\alpha})}{p^{\alpha(2+2i\epsilon_1r_1+2i\epsilon_2r_2)}}\\
\times \sum_{A,B=0}^{2}C(A,B)(p^A)^{t+s-i\epsilon_1r_1-i\epsilon_2r_r}
\sum_{0\leq \beta \leq \nu-A}\frac{(p^B,p^{\alpha+\beta})^{1+2i\epsilon_2r_2}}{p^{B(1+i\epsilon_2r_2)}p^{\beta(1+2i\epsilon_1r_1+2i\epsilon_2r_2)}}.
\end{multline*}
The sum over $\beta$ can be decomposed in the following way:
\begin{multline*}
\sum_{0\leq \beta \leq \nu-A}\frac{(p^B,p^{\alpha+\beta})^{1+2i\epsilon_2r_2}}{p^{B(1+i\epsilon_2r_2)}p^{\beta(1+2i\epsilon_1r_1+2i\epsilon_2r_2)}}=
\sum_{\substack{0\leq \beta \leq \nu-A\\ B \leq \alpha+\beta}}\frac{(p^B)^{i\epsilon_2r_2}}{p^{\beta(1+2i\epsilon_1r_1+2i\epsilon_2r_2)}}\\+\sum_{\substack{0\leq \beta \leq \nu-A\\ B>\alpha+\beta}}\frac{(p^{\alpha})^{1+2i\epsilon_2r_2}}{p^{B(1+i\epsilon_2r_2)}p^{\beta(2i\epsilon_1r_1)}}=
\sum_{\substack{0\leq \beta \leq \nu-A}}\frac{(p^B)^{i\epsilon_2r_2}}{p^{\beta(1+2i\epsilon_1r_1+2i\epsilon_2r_2)}}\\-\sum_{\substack{0\leq \beta \leq \nu-A\\ B>\alpha+\beta}}\frac{(p^B)^{i\epsilon_2r_2}}{p^{\beta(1+2i\epsilon_1r_1+2i\epsilon_2r_2)}}+\sum_{\substack{0\leq \beta \leq \nu-A\\ B>\alpha+\beta}}\frac{(p^{\alpha})^{1+2i\epsilon_2r_2}}{p^{B(1+i\epsilon_2r_2)}p^{\beta(2i\epsilon_1r_1)}}
.
\end{multline*}
The first sum does not contribute to $\Phi(s,t)$ because
\begin{equation*}
\sum_{0 \leq \beta\leq \nu-A}\frac{1}{p^{\beta(1+2i\epsilon_1r_1+2i\epsilon_2r_2)}}=\left(1-\frac{1}{p^{1+2i\epsilon_1r_1+2i\epsilon_2r_2}}\right)^{-1}\left(1+O\left(\frac{1}{q}\right)\right)
\end{equation*}
and
\begin{equation*}
\sum_{A,B=0}^{2}C(A,B)(p^A)^{t+s-i\epsilon_1r_1-i\epsilon_2r_2}(p^B)^{i\epsilon_2r_2}=0.
\end{equation*}
Therefore,
\begin{multline*}
\Phi(s,t)=\hat{q}^{2i\epsilon_1r_1+2i\epsilon_2r_2}\frac{\zeta_q(1+t+s+i\epsilon_1r_1+i\epsilon_2r_2)}{\zeta_q(2+2i\epsilon_1r_1+2i\epsilon_2r_2)}\\
\times \zeta(1-t+s+i\epsilon_1r_1+i\epsilon_2r_2)\zeta(1-t-s+i\epsilon_1r_1+i\epsilon_2r_2)\\
\times \zeta_q(1+t-s+i\epsilon_1r_1+i\epsilon_2r_2)
\sum_{A,B=0}^{2}C(A,B)(p^A)^{t+s-i\epsilon_1r_1-i\epsilon_2r_r}\\
\times \sum_{\alpha \geq 0} \frac{\mu(p^{\alpha})}{p^{\alpha(2+2i\epsilon_1r_1+2i\epsilon_2r_2)}}
\sum_{\substack{0\leq \beta \leq \nu-A\\ B>\alpha+\beta}}\left(\frac{-(p^B)^{i\epsilon_2r_2}}{p^{\beta(1+2i\epsilon_1r_1+2i\epsilon_2r_2)}}+\frac{(p^{\alpha})^{1+2i\epsilon_2r_2}}{p^{B(1+i\epsilon_2r_2)}p^{\beta(2i\epsilon_1r_1)}}\right).
\end{multline*}
For each fixed $B$ the sum over $A$ can be evaluated using table \ref{tab:coeff}:
\begin{equation*}
\sum_{A=0}^{2}C(A,0)(p^A)^{t+s-i\epsilon_1r_1-i\epsilon_2r_2}=
(1-p^{t+s+1-i\epsilon_1r_1-i\epsilon_2r_2})(1-p^{t-s+1-i\epsilon_1r_1-i\epsilon_2r_2}),
\end{equation*}
\begin{multline*}
\sum_{A=0}^{2}C(A,1)(p^A)^{t+s-i\epsilon_1r_1-i\epsilon_2r_2}=-(p^{ir_2}+p^{-ir_2})
(1-p^{t+s+1-i\epsilon_1r_1-i\epsilon_2r_2})\\
\times (1-p^{t-s+1-i\epsilon_1r_1-i\epsilon_2r_2}),
\end{multline*}
\begin{equation*}
\sum_{A=0}^{2}C(A,2)(p^A)^{t+s-i\epsilon_1r_1-i\epsilon_2r_2}=
(1-p^{t+s+1-i\epsilon_1r_1-i\epsilon_2r_2})(1-p^{t-s+1-i\epsilon_1r_1-i\epsilon_2r_2}).
\end{equation*}
 Since $B=0,1,2$, the  requirement $B>\alpha+\beta$ is satisfied in four cases
\begin{equation*}
(B,\alpha,\beta)=\{(1,0,0), (2,0,0), (2,1,0), (2,0,1)\}.
\end{equation*}
Thus,
\begin{multline*}
\Phi(s,t)=\hat{q}^{2i\epsilon_1r_1+2i\epsilon_2r_2}\\
\times \frac{\zeta_q(1+t+s+i\epsilon_1r_1+i\epsilon_2r_2)\zeta_q(1+t-s+i\epsilon_1r_1+i\epsilon_2r_2)}{\zeta_q(2+2i\epsilon_1r_1+2i\epsilon_2r_2)}\\
\times \zeta_q(1-t+s+i\epsilon_1r_1+i\epsilon_2r_2)\zeta_q(1-t-s+i\epsilon_1r_1+i\epsilon_2r_2)\\
\times \left[(p^{ir_2}+\frac{1}{p^{ir_2}})(p^{i\epsilon_2r_2}-\frac{1}{p^{1+i\epsilon_2r_2}})-p^{2i\epsilon_2r_2}+\frac{1}{p^{2+2i\epsilon_2r_2}}\right.\\
\left. +\frac{1}{p^{2+2i\epsilon_1r_1}}-\frac{1}{p^{3+2i\epsilon_1r_1+2i\epsilon_2r_2}}-\frac{1}{p^{1+2i\epsilon_1r_1}}+\frac{1}{p^{2+2i\epsilon_1r_1+2i\epsilon_2r_2}}\right].
\end{multline*}
Simplifying, we have
\begin{multline*}
\Phi(s,t)=\frac{\phi(q)}{q}\hat{q}^{2i\epsilon_1r_1+2i\epsilon_2r_2}\left(1-\frac{1}{p^{1+2i\epsilon_1r_1}}\right)\left(1-\frac{1}{p^{1+2i\epsilon_2r_2}}\right)\\
\times \frac{\zeta_q(1+t+s+i\epsilon_1r_1+i\epsilon_2r_2)\zeta_q(1+t-s+i\epsilon_1r_1+i\epsilon_2r_2)}{\zeta_q(2+2i\epsilon_1r_1+2i\epsilon_2r_2)}\\
\times \zeta_q(1-t+s+i\epsilon_1r_1+i\epsilon_2r_2)\zeta_q(1-t-s+i\epsilon_1r_1+i\epsilon_2r_2) .
\end{multline*}
Substituting this result in \eqref{eqE}, we prove theorem \ref{mainOOD}.

\section{Off-off-diagonal term: asymptotic evaluation }\label{asympOOD}
\begin{theorem}
Up to a negligible error term, we have
\begin{multline}
M^{OOD}=\frac{\phi(q)}{q}\sum_{\epsilon_1,\epsilon_2=\pm 1}\frac{\zeta_q(1+2i\epsilon_1r_1)\zeta_q(1+2i\epsilon_2r_2)}{\zeta_q(2+2i\epsilon_1r_1+2i\epsilon_2r_2)}\\
\times \prod_{\epsilon_3, \epsilon_4 =\pm 1} \zeta_q(1+\epsilon_3 t_1+\epsilon_4 t_2+i\epsilon_1r_1+i\epsilon_2r_2)\hat{q}^{-2t_1-2t_2+2i\epsilon_1r_1+2i\epsilon_2r_2}\\
\times \frac{\Gamma(k/2-t_1+i\epsilon_1r_1)\Gamma(k/2-t_2 +i\epsilon_2r_2)}{\Gamma(k/2+t_1-i\epsilon_1r_1)\Gamma(k/2+t_2-i\epsilon_2r_2)}.
\end{multline}

\end{theorem}
\begin{proof}
Consider
\begin{multline*}
M^{OOD}=\frac{\phi(q)}{q}\sum_{\epsilon_1,\epsilon_2=\pm 1}\frac{\zeta_q(1+2i\epsilon_1r_1)\zeta_q(1+2i\epsilon_2r_2)}{\zeta_q(2+2i\epsilon_1r_1+2i\epsilon_2r_2)}\hat{q}^{-2t_1-2t_2+2i\epsilon_1r_1+2i\epsilon_2r_2}\\
\times \frac{1}{(2\pi i)^2}\int_{\Re{t}=k/2+0.7}\int_{\Re{s}=k/2-0.4}I_{\epsilon_1,\epsilon_2}(s,t)\frac{2sds}{s^2-t_{1}^{2}}\frac{2tdt}{t^2-t_{2}^{2}},
\end{multline*}
where
\begin{multline*}
I_{\epsilon_1,\epsilon_2}(s,t)=
\frac{P_r(s)P_r(t)}{P_r(t_1)P_r(t_2)}\zeta_q(1+t+s+i\epsilon_1r_1+i\epsilon_2r_2)\zeta_q(1+t-s+i\epsilon_1r_1+i\epsilon_2r_2)\\
\times \zeta_q(1-t+s+i\epsilon_1r_1+i\epsilon_2r_2)\zeta_q(1-t-s+i\epsilon_1r_1+i\epsilon_2r_2)\\
\times \frac{\Gamma(k/2+s +i\epsilon_1r_1)\Gamma(k/2+t +i\epsilon_2r_2)\Gamma(k/2-s +i\epsilon_1r_1)\Gamma(k/2-t +i\epsilon_2r_2)}{\Gamma(k/2+t_1+ir_1)\Gamma(k/2+t_1-ir_1)\Gamma(k/2+t_2+ir_2)\Gamma(k/2+t_2-ir_2)}.
\end{multline*}
The function $I_{\epsilon_1,\epsilon_2}(s,t)$ is even in both $s$ and $t$. Therefore,
\begin{multline*}
4\frac{1}{(2\pi i)^2}\int_{\Re{t}=k/2+0.7}\int_{\Re{s}=k/2-0.4}I_{\epsilon_1,\epsilon_2}(s,t)\frac{2sds}{s^2-t_{1}^{2}}\frac{2tdt}{t^2-t_{2}^{2}}\\ =\res_{\substack{s=t_1\\ t=t_2}}{I(s,t)\frac{2s}{s^2-t_{1}^{2}}\frac{2t}{t^2-t_{2}^{2}}}+\res_{\substack{s=t_1\\ t=-t_2}}{I(s,t)\frac{2s}{s^2-t_{1}^{2}}\frac{2t}{t^2-t_{2}^{2}}}\\+\res_{\substack{s=-t_1\\ t=t_2}}{I(s,t)\frac{2s}{s^2-t_{1}^{2}}\frac{2t}{t^2-t_{2}^{2}}}+\res_{\substack{s=-t_1\\t=-t_2}}{I(s,t)\frac{2s}{s^2-t_{1}^{2}}\frac{2t}{t^2-t_{2}^{2}}}.
\end{multline*}
Each of the four given residues has the same value.
Computing the residue yields the assertion of our theorem.
\end{proof}

\begin{theorem}
Up to a negligible error, we have
\begin{equation}
\lim_{(\uple{t}, \uple{r})\rightarrow (0,0)} M^{OOD}
=\frac{1}{(2 \pi i)^2}\int_{\Re{t}=k/2+0.7}\int_{\Re{s}=k/2-0.4}g(s,t)\frac{2ds}{s}\frac{2dt}{t},
\end{equation}
where
\begin{multline*}
g(s,t)=\left(\frac{\phi(q)}{q}\right)^3\frac{1}{\zeta_q(2)}\frac{P_r(s)P_r(t)}{P_r(0)^2}\prod_{\epsilon_1,\epsilon_2=\pm 1} \zeta_q(1+\epsilon_1 t +\epsilon_2 s) \\
\times \frac{\Gamma(k/2+\epsilon_1 t )\Gamma(k/2+\epsilon_2 s )}{\Gamma(k/2)^4} \Biggl[(2 \log{\hat{q}}+\gamma)^2+\sum_{\epsilon_1,\epsilon_2=\pm 1}  \frac{\zeta_{q}^{''}}{\zeta_q}(1+\epsilon_1 t +\epsilon_2 s)\\+2\sum_{\substack{\epsilon_1,\epsilon_2,\epsilon_3, \epsilon_4= \pm 1\\ (\epsilon_1,\epsilon_2)\neq (\epsilon_3,\epsilon_4)}}\frac{\zeta_{q}^{'}}{\zeta_q}(1+\epsilon_1 t+\epsilon_2 s)\frac{\zeta_{q}^{'}}{\zeta_q}(1+\epsilon_3 t+\epsilon_4 s)
+(2 \log{\hat{q}}+\gamma)\\
\times \biggl(4\frac{\zeta_{q}^{'}}{\zeta_{q}}(2) -2\sum_{\epsilon_1,\epsilon_2=\pm 1} \frac{\zeta_{q}^{'}}{\zeta_q}(1+\epsilon_1 t +\epsilon_2 s)-\sum_{\epsilon=\pm 1} \frac{\Gamma^{'}}{\Gamma}(k/2 +\epsilon s)\\-\sum_{\epsilon=\pm 1} \frac{\Gamma^{'}}{\Gamma}(k/2 +\epsilon t) \biggr)
 +\sum_{\epsilon_1,\epsilon_2=\pm 1} \frac{\zeta_{q}^{'}}{\zeta_q}(1+\epsilon_1 t+\epsilon_2 s)\\
\times \left( -4\frac{\zeta_{q}^{'}}{\zeta_q}(2)+\sum_{\epsilon=\pm 1} \frac{\Gamma^{'}}{\Gamma}(k/2 +\epsilon t)+\sum_{\epsilon=\pm 1} \frac{\Gamma^{'}}{\Gamma}(k/2 +\epsilon s)\right)
  \\ +\sum_{\epsilon=\pm 1} \frac{\Gamma^{'}}{\Gamma}(k/2 +\epsilon t) \sum_{\epsilon=\pm 1} \frac{\Gamma^{'}}{\Gamma}(k/2 +\epsilon s) -4\frac{\zeta_{q}^{''}}{\zeta_q}(2)+8\left(\frac{\zeta_{q}^{'}}{\zeta_q} (2)\right)^2
\\-2\frac{\zeta_{q}^{'}}{\zeta_q} (2)\left(\sum_{\epsilon=\pm 1} \frac{\Gamma^{'}}{\Gamma}(k/2 +\epsilon t)\sum_{\epsilon=\pm 1} \frac{\Gamma^{'}}{\Gamma}(k/2 +\epsilon s) \right)\Biggr].
\end{multline*}

\end{theorem}

\begin{cor}
The off-off-diagonal term at the critical point is a polynomial in $\log{q}$ of order $2$.
\end{cor}
\begin{proof}
First, we let $t_1,t_2 \rightarrow 0$. Then
\begin{multline*}
\lim_{\uple{t}\rightarrow 0} M^{OOD}= \frac{\phi(q)}{q}\sum_{\epsilon_1,\epsilon_2=\pm 1}\frac{\zeta_q(1+2i\epsilon_1r_1)\zeta_q(1+2i\epsilon_2r_2)}{\zeta_q(2+2i\epsilon_1r_1+2i\epsilon_2r_2)} \hat{q}^{2i\epsilon_1r_1+2i\epsilon_2r_2}\\ \times \frac{1}{(2\pi i)^2}\int_{\Re{t}=k/2+0.7}\int_{\Re{s}=k/2-0.4}I_{\epsilon_1,\epsilon_2}(s,t)\frac{2ds}{s}\frac{2dt}{t},
\end{multline*}
where
\begin{multline*}
I_{\epsilon_1,\epsilon_2}(s,t)=
\frac{P_r(s)P_r(t)}{P_r(0)^2} \zeta_q(1+t+s+i\epsilon_1r_1+i\epsilon_2r_2)\\
\times \zeta_q(1+t-s+i\epsilon_1r_1+i\epsilon_2r_2)\zeta_q(1-t+s+i\epsilon_1r_1+i\epsilon_2r_2)\zeta_q(1-t-s+i\epsilon_1r_1+i\epsilon_2r_2)\\
\times \frac{\Gamma(k/2+s +i\epsilon_1r_1)\Gamma(k/2+t +i\epsilon_2r_2)\Gamma(k/2-s +i\epsilon_1r_1)\Gamma(k/2-t +i\epsilon_2r_2)}{\Gamma(k/2+ir_1)\Gamma(k/2-ir_1)\Gamma(k/2+ir_2)\Gamma(k/2-ir_2)}.
 \end{multline*}
Let
\begin{multline*}
f(r_1,r_2):=\frac{\phi(q)}{q}\frac{P_r(s)P_r(t)}{P_r(0)^2}\frac{\zeta_q(1+t+s+ir_1+ir_2)}{\zeta_q(2+2ir_1+2ir_2)}\\
\times \zeta_q(1+t-s+ir_1+ir_2)\zeta_q(1-t+s+ir_1+ir_2)\zeta_q(1-t-s+ir_1+ir_2)\\
\times \frac{\Gamma(k/2+s +ir_1)\Gamma(k/2+t +ir_2)\Gamma(k/2-s +ir_1)\Gamma(k/2-t +ir_2)}{\Gamma(k/2+ir_1)\Gamma(k/2-ir_1)\Gamma(k/2+ir_2)\Gamma(k/2-ir_2)}.
\end{multline*}
Consider
\begin{multline*}\label{eq:gst}
g(s,t)=\lim_{r_1\rightarrow 0}\lim_{r_2\rightarrow 0}\sum_{\epsilon_1,\epsilon_2=\pm 1}\zeta_q(1+2i\epsilon_1r_1)\zeta_q(1+2i\epsilon_2r_2)\hat{q}^{2i\epsilon_1r_1+2i\epsilon_2r_2}\\
\times f(\epsilon_1 r_1,\epsilon_2r_2)=
\left(\frac{\phi(q)}{q}\right)^2 \biggl[(2\log{\hat{q}}+\gamma)^2f(0,0)\\+i(2\log{\hat{q}}+\gamma)\left(\frac{\partial f}{\partial r_1}(0,0)+\frac{\partial f}{\partial r_2}(0,0)\right)-\frac{\partial^2f}{\partial r_1 \partial r_2}(0,0)\biggr].
\end{multline*}
Here
\begin{equation*}
\frac{\partial f}{\partial r_1}(0,0)=-if(0,0)\left(2\frac{\zeta_{q}^{'}}{\zeta_{q}}(2) -\sum \frac{\zeta_{q}^{'}}{\zeta_q}(1 \pm t\pm s)-
\sum \frac{\Gamma^{'}}{\Gamma}(k/2 \pm s)\right),
\end{equation*}
\begin{equation*}
\frac{\partial f}{\partial r_2}(0,0)=-if(0,0)\left(2\frac{\zeta_{q}^{'}}{\zeta_{q}}(2) -\sum \frac{\zeta_{q}^{'}}{\zeta_q}(1 \pm t\pm s)-
\sum \frac{\Gamma^{'}}{\Gamma}(k/2 \pm t)\right),
\end{equation*}
\begin{multline*}
\frac{\partial^2f}{\partial r_1 \partial r_2}(0,0)=-f(0,0)\biggl[
 \sum \frac{\zeta_{q}^{''}}{\zeta_q}(1\pm t\pm s)\\+2\sum^{*}\frac{\zeta_{q}^{'}}{\zeta_q}(1\pm t\pm s)\frac{\zeta_{q}^{'}}{\zeta_q}(1\pm t\pm s)  +\sum \frac{\zeta_{q}^{'}}{\zeta_q}(1\pm t\pm s)\\
\times\left( -4\frac{\zeta_{q}^{'}}{\zeta_q}(2)+\sum \frac{\Gamma^{'}}{\Gamma}(k/2 \pm t)+\sum \frac{\Gamma^{'}}{\Gamma}(k/2 \pm s)\right) \\+\sum \frac{\Gamma^{'}}{\Gamma}(k/2 \pm t) \sum \frac{\Gamma^{'}}{\Gamma}(k/2 \pm s)  -4\frac{\zeta_{q}^{''}}{\zeta_q}(2)+8\left(\frac{\zeta_{q}^{'}}{\zeta_q} (2)\right)^2\\
-2\frac{\zeta_{q}^{'}}{\zeta_q} (2)\left(\sum \frac{\Gamma^{'}}{\Gamma}(k/2 \pm t)\sum \frac{\Gamma^{'}}{\Gamma}(k/2 \pm s) \right)\biggr].
\end{multline*}
Then
\begin{equation*}
\lim_{(\uple{t},\uple{r})\rightarrow (0,0)} M^{OOD}=\frac{1}{(2 \pi i)^2}\int_{\Re{t}=k/2+0.7}\int_{\Re{s}=k/2-0.4}g(s,t)\frac{2ds}{s}\frac{2dt}{t}.
\end{equation*}
The function $g(s,t)$ is even in both variables $s$ and $t$. Therefore,
$$\lim_{(\uple{t},\uple{r})\rightarrow (0,0)} M^{OOD}=\frac{1}{4}\res_{s=t=0}\frac{4g(s,t)}{st}=\res_{t=0}\frac{g(0,t)}{t}.$$
To find the order of the leading term, we replace all $\zeta(1\pm t)$ by $\frac{1}{\pm t}$. Let $$r(t):=\frac{P_r(t)}{P_r(0)}\frac{\Gamma(k/2+t)\Gamma(k/2-t)}{\Gamma(k/2)^2}.$$ Then
\begin{multline*}
\left(\frac{\phi(q)}{q}\right)^7\frac{1}{\zeta_q(2)}\res_{t=0}\frac{r(t)}{t^5}\left((\log{q})^2+\frac{4}{t^2}\right)\\=\left(\frac{\phi(q)}{q}\right)^7\frac{1}{\zeta_q(2)}\frac{1}{6!}(4r^{(6)}(0)+30 r^{(4)}(0)(\log{q})^2).
\end{multline*}
Therefore, $\lim_{(\uple{t},\uple{r})\rightarrow (0,0)} M^{OOD}$ is a polynomial in $\log{q}$ of order $2$.

\end{proof}

\subsection*{Acknowledgements}
 I would like to thank the ALGANT program and my supervisors (Andrew Granville, Laurent Habsieger, Giuseppe Molteni and Guillaume Ricotta) for giving me the opportunity to work on this subject. I am grateful to Brian Conrey and Dmitry Frolenkov for interesting and helpful discussions.
 Finally, I thank the referee for reading the paper very carefully and providing valuable comments.

\begin{appendices}
\section{Bessel functions}
\begin{lemma}(\cite{MicBes}, Lemma C.1)
Let $z>0$ and $v \in \mathbb{C}$. Then

\begin{equation}\label{recBess}
(z^{v}J_{v}(z))^{'}=z^vJ_{v-1}(z),
\end{equation}

\begin{equation}
(z^{v}Y_{v}(z))^{'}=z^vY_{v-1}(z),
\end{equation}

\begin{equation}\label{RecK}
(z^{v}K_{v}(z))^{'}=-z^vK_{v-1}(z).
\end{equation}
\end{lemma}

\begin{lemma}(\cite{MicBes}, Lemma C.2)
For $z>0$ and $j \geq 0$ we have
\begin{equation}\label{eq:Jbes}
\frac{z^{j}}{(1+z)^j}J_{v}^{(j)}(z) \ll_{j,v} \frac{z^{\Re{v}}}{(1+z)^{\Re{v}+1/2}},
\end{equation}

\begin{equation}\label{YBES}
\frac{z^{j}}{(1+z)^j}Y_{0}^{(j)}(z) \ll_{j} \frac{(1+|\log{z}|)}{(1+z)^{1/2}},
\end{equation}

\begin{equation}\label{BESSKJ}
\frac{z^{j}}{(1+z)^j}K_{v}^{(j)}(z) \ll_{j,v} \frac{e^{-z}(1+|\log{z}|)}{(1+z)^{1/2}}\text{ if } \Re{v}=0.
\end{equation}
\end{lemma}

\begin{lemma}\label{BesselInt} (\cite{watson}, page 149)
 Assume that $\Re{(\mu_1+\mu_2+1)}>\Re{(2s)}>0.$ Then
\begin{multline}
 \int_{0}^{\infty}\frac{J_{\mu_1}(z)J_{\mu_2}(z)}{z^{2s}}dz=\frac{1}{2^{2s}}\frac{\Gamma(2s)}{\Gamma(-\mu_1/2+\mu_2/2+s+1/2)}\\
\times \frac{\Gamma(\mu_1/2+\mu_2/2-s+1/2)}{\Gamma(\mu_1/2+\mu_2/2+s+1/2)\Gamma(\mu_1/2-\mu_2/2+s+1/2)}.
\end{multline}
\end{lemma}
\begin{lemma}\label{lemmaBH}(\cite{BH}, lemma 3)
Let $F:(0,\infty)\rightarrow \mathbb{C}$ be a smooth function of compact support. For $s \in \mathbb{C}$ let $B_s$ denote one of $J_s$, $Y_s$ or $K_s$. Then for $\alpha>0$ and $j \in \mathbb{N}$ we have
\begin{multline}
\int_{0}^{\infty}F(x)B_s(\alpha\sqrt{x})dx\\=\pm \left(\frac{2}{\alpha}\right)^{j}\int_{0}^{\infty}
\frac{\partial^{j}}{\partial x^{j}}(F(x)x^{-s/2})x^{\frac{s+j}{2}}B_{s+j}(\alpha\sqrt{x})dx.
\end{multline}
\end{lemma}

\begin{lemma}(\cite{olver2010nist}, equations 10.6.7 and 10.29.5)
For $k=0,1,2, \ldots,$
\begin{equation}\label{jbesder}
J_{s}^{(k)}(z)=\frac{1}{2^k}\sum_{n=0}^{k}(-1)^n \binom{k}{n}J_{s-k+2n}(z),
\end{equation}
\begin{multline}
e^{s\pi i}K_{s}^{(k)}(z)=\frac{1}{2^k} \left(e^{(s-k)\pi i}K_{s-k}(z)+ \binom{k}{1}e^{(s-k+2)\pi i}K_{s-k+2}(z) \right.\\ \left. +\binom{k}{2}e^{(s-k+4)\pi i}K_{s-k+4}(z)+ \ldots+e^{(s+k)\pi i}K_{s+k}(z) \right).
\end{multline}
\end{lemma}

\section{ Mellin transforms}
\begin{lemma}\label{M1}(\cite{Ob}, 2.19, page 15)
Let $\phi(x)=(b+ax)^{-v}$. Then for $0<\Re{z}<v$

\begin{equation}
\int_{0}^{\infty}\phi(x)x^{z-1}dx=(b/a)^zb^{-v}\frac{\Gamma(z)\Gamma(v-z)}{\Gamma(v)}.
\end{equation}

\end{lemma}

\begin{lemma}\label{M2}(\cite{Ob}, 2.20, page 16)
Let $\Re{v}>-1$ and $$\phi(x)=\begin{cases}(a-x)^{v}&\mbox{if }x<a\\0 &\mbox{if }x>a
\end{cases}.$$
 Then for $\Re{z}>0$

\begin{equation}
\int_{0}^{\infty}\phi(x)x^{z-1}dx=a^{v+z}\frac{\Gamma(v+1)\Gamma(z)}{\Gamma(v+z+1)}.
\end{equation}

\end{lemma}

\begin{lemma}\label{M3}(\cite{Ob}, 2.21, page 16)
Let $\Re{v}>-1$ and $$\phi(x)=\begin{cases}(x-a)^{v}&\mbox{if }x>a\\0 &\mbox{if }x<a
\end{cases}.$$
 Then for $\Re{z}<-\Re{v}$

\begin{equation}
\int_{0}^{\infty}\phi(x)x^{z-1}dx=a^{v+z}\frac{\Gamma(-v-z)\Gamma(v+1)}{\Gamma(1-z)}.
\end{equation}

\end{lemma}

\begin{lemma}(\cite{higher}, p.21)
For $x>0$,
\begin{equation}
J_{k-1}(x)=\frac{1}{2}\frac{1}{2\pi i}\int_{(\sigma)}\left(\frac{x}{2}\right)^{-s}\frac{\Gamma(s/2+k/2-1/2)}{\Gamma(-s/2+k/2+1/2)}ds,
\end{equation}
where $-k+1<\sigma<1$.
\end{lemma}
Changing variable $-s:=k-1+2z$, we obtain
\begin{lemma}
\begin{equation}\label{Mel1}
J_{k-1}(x)=-\frac{1}{2\pi i}\int_{(\sigma)}\frac{\pi}{\Gamma(1+z)\Gamma(k+z)\sin{(\pi z)}}\left(\frac{x}{2}\right)^{k-1+2z}dz,
\end{equation}
where
$-k/2<\sigma<0$.
\end{lemma}
Let
\begin{equation}
\gamma(u,v):=\frac{2^{2u-1}}{\pi}\Gamma(u+v-1/2)\Gamma(u-v+1/2).
\end{equation}
\begin{lemma}(\cite{Kuz}, p. 89)
\begin{equation}
\int_{0}^{\infty}k_0(x,v)x^{w-1}dx=\gamma(w/2,v)\cos{(\pi w/2)},
\end{equation}
\begin{equation}
\int_{0}^{\infty}k_1(x,v)x^{w-1}dx=\gamma(w/2,v)\sin{(\pi v)}.
\end{equation}
\end{lemma}
\begin{cor}

\begin{equation}\label{Mel3}
k_1(x,1/2+ir)=\frac{\sin{(\pi (1/2+ir))}}{2\pi i}\int_{(0.7)}x^{-2\beta}\gamma(\beta,1/2+ir)2d\beta,
\end{equation}
\begin{equation}\label{Mel2}
k_0(x,1/2+ir)=\frac{1}{2\pi i}\int_{(*)}x^{-2\beta}\gamma(\beta,1/2+ir)\cos{(\pi \beta)}2d\beta,
\end{equation}
where the contour of integration $(*)$ can be taken as $\Re{\beta}=-1$ except the area $|\Im{\beta}|<1$, where it crosses the real axis at $\Re{\beta}>0$.
\end{cor}
\end{appendices}


\begin{thebibliography}{HD}




\normalsize
\baselineskip=17pt


\bibitem[A]{Akbary} A. Akbary,
\emph{Non-vanishing of weight {$k$} modular {$L$}-functions with large level},
J. Ramanujan Math. Soc. {14(1)} (1999), 37--54.

\bibitem[AL]{atkin-lehner} A. O.  L. Atkin and J. Lehner,
\emph{Hecke operators on {$\Gamma _{0}(m)$}},
Math. Ann. {185} (1970), 134--160.

\bibitem[BE]{higher} H. Bateman and A. Erdelyi,
\emph{Higher Transcendental Functions},
Number v. {2} in Higher Transcendental Functions,
Dover Publications, Incorporated, 2007.

\bibitem[BH]{BH} V. Blomer and G. Harcos,
\emph{A hybrid asymptotic formula for the second moment of {R}ankin-{S}elberg {$L$}-functions},
Proc. London Math. Soc. {105} (2012), 473--505.

\bibitem[CFKRS]{Conrey} J.B. Conrey, D.W. Farmer, J. P. Keating, M. O. Rubinstein and N.C. Snaith,
\emph{Integral moments of {$L$}-functions},
Proc. London Math. Soc. (3), {91-1} (2005), 33-104.

\bibitem[DI]{DesIw} J.-M. Deshouillers and H. Iwaniec,
\emph{Kloosterman sums and {F}ourier coefficients of cusp forms},
Invent. Math. {70-2} (1982/83), 219--288.

\bibitem[D]{Du} W. Duke,
\emph{The critical order of vanishing of automorphic {$L$}-functions with large level},
Invent. Math. {119-1} (1995), 165--174.

\bibitem[DFI]{DFI} W. Duke, J. Friedlander, H. Iwaniec,
\emph{Bounds for automorphic {$L$}-functions {II}},
Invent. Math. {115-2} (1994), 219--239.

\bibitem[DFI2]{DFIQ} W. Duke, J. Friedlander, H. Iwaniec,
\emph{A quadratic divisor problem},
Invent. Math. {115-2} (1994), 209--217.

\bibitem[DFI3]{DFI3} W. Duke, J. Friedlander, H. Iwaniec,
\emph{Erratum. Bounds for automorphic {$L$}-functions {II}},
Invent. Math. {140} (2000), 227--242.

\bibitem[Ich]{Ichihara} Y. Ichihara,
\emph{The first moment of {$L$}-functions of primitive forms on {$\Gamma_0(p^\alpha)$} and a basis of old forms},
J. Number theory {131-2} (2011), 343--362.

\bibitem[Iw]{Iw1} H. Iwaniec,
\emph{Topics in classical automorphic forms},
volume {17} of Graduate Studies in Mathematics,
 American Mathematical Society, Providence, RI, 1997.

\bibitem[IK]{IwKow}H. Iwaniec and E. Kowalski,
\emph{Analytic number theory},
 volume {53} of American Mathematical Society Colloquium Publications,
American Mathematical Society, Providence, RI, 2004.

\bibitem[ILS]{ILS} H. Iwaniec, W. Luo and P. Sarnak,
 \emph{Low lying zeros of families of {$L$}-functions},
Inst. Hautes Etudes Sci. Publ. Math. {91} (2000), 55--131.

\bibitem[IS]{IwSa} H. Iwaniec and P. Sarnak,
\emph{ The non-vanishing of central values of automorphic L-functions and Landau-Siegel zeros},
 Israel J. Math. {120} part A  (2000), 155--177.

\bibitem[KS]{KS} H. H. Kim (with appendices by D. Ramakrishnan, H. H. Kim and P. Sarnak),
\emph{ Functoriality for the
exterior square of $GL_4$ and the symmetric square of $GL_2$}, J. Amer. Math. Soc. {16} (2003), 139--183.

\bibitem[KM]{KMV5} E. Kowalski and P. Michel,
\emph{A lower bound for the rank of  {$J_0(q)$}},
 Acta Arith.  {94-4}  (2000), 303--343.

\bibitem[KMV]{KMV} E. Kowalski, P. Michel and J. VanderKam,
\emph{Mollification of the fourth moment of automorphic $L$-functions and arithmetic
applications},
Invent. Math.  {142-1} (2000), 95--151.

\bibitem[KMV2]{MicBes} E. Kowalski, P. Michel and J. VanderKam,
\emph{ Rankin-Selberg $L$-functions in the level aspect},
 Duke Mathematical Journal {114-1}  (2002),  123--191.

\bibitem[K]{Kuz} N. V. Kuznetsov,
\emph{Trace formulas and some applications in analytic number theory},
 Far East Division of the Russian Academy
of Sciences, Dalnauka, Vladivostok, 2003.

\bibitem[M]{MR2019017} P. Michel,
\emph{ Familles de fonctions L de formes automorphes et applications},
 J. Theor. Nombres Bordeaux {15-1} (2003), 275--307.

\bibitem[Ob]{Ob} F. Oberhettinger,
\emph{ Tables of Mellin transforms},
 Springer-Verlag, New York-Heidelberg, 1974.

\bibitem[Ol]{olver2010nist} F.W.J. Olver,
\emph{NIST Handbook of Mathematical Functions},
National Institute of Standards and Technology (U.S.),
Cambridge University Press, 2010.

\bibitem[RR]{RiRo} G. Ricotta and E. Royer,
\emph{ Statistics for low-lying zeros of symmetric power $L$-functions in the level aspect},
 Forum Math.
{23-5} (2011), 969--1028.

\bibitem[R]{R}D. Rouymi,
\emph{ Formules de trace et non-annulation de fonctions $L$ automorphes au niveau $p^v$},
 Acta Arith. {147-1} (2011), 1--32.

\bibitem[Ro]{royer} E. Royer,
\emph{ Sur les fonctions $L$ de formes modulaires},
 PhD thesis, Universit\'{e} de Paris-Sud, 2001.

\bibitem[T]{titc} E. C. Titchmarsh,
\emph{ The theory of the Riemann zeta-function},
The Clarendon Press, Oxford University Press, New York,
second edition, 1986.

\bibitem[Wa]{watson} G. N. Watson,
\emph{ A Treatise on the Theory of Bessel Functions},  Cambridge University Press, Cambridge, England, 1944.

\bibitem[We]{weil} A. Weil,
\emph{ On some exponential sums},
Proc. Nat. Acad. Sci. U. S. A.  {34} (1948),  204--207.


\end{thebibliography}
\end{document}